%% file: thermocline.tex
\documentclass[11pt]{amsart}
\usepackage{graphicx, amssymb}

\usepackage{enumerate}
\numberwithin{equation}{section}
\usepackage[french,english]{babel}
\usepackage[utf8]{inputenc}

\newtheorem{Thm}{Theorem}[section]
\newtheorem{Rmk}[Thm]{Remark}
\newtheorem{Prop}[Thm]{Proposition}
\newtheorem{Cor}[Thm]{Corollary}
\newtheorem{Lem}[Thm]{Lemma}
\newtheorem{Lemma}{Lemma}
\newtheorem{Def}[Thm]{Definition}

\def\bR {\mathbf{R}}

\def\bT {\mathbf{T}}
\def\bZ {\mathbf{Z}}
\def\Z {\mathbf{Z}}

\def\cF {\mathcal{F}}

\def\eps {{\epsilon}}

\def\om {{\omega}}

\def\d {{\partial}}

\newcommand{\Div}{\operatorname{div}}
\newcommand{\rot}{\operatorname{rot}}

\newcommand{\sgn}{\operatorname{sign}}

\newcommand{\Supp}{\operatorname{supp}}

\newcommand{\ba}{\begin{aligned}}
\newcommand{\ea}{\end{aligned}}

\newcommand{\be}{\begin{equation}}
\newcommand{\ee}{\end{equation}}

\newcommand{\ubl}{u^{BL}}
\newcommand{\pbl}{p^{BL}}
\newcommand{\Pbl}{P^{BL}}
\newcommand{\Ubl}{U^{BL}}
\newcommand{\uint}{u^{int}}
\newcommand{\vint}{v^{int}}
\newcommand{\ustat}{u^{stat}}
\newcommand{\pstat}{p^{stat}}

\newcommand{\Ph}{\mathbb P_{2D}}
\newcommand{\bt}{\bar \theta}
\newcommand{\tbl}{\theta^{BL}}
\newcommand{\tapp}{\theta^{app}}

\begin{document}
\title[$\beta$-plane model for rotating fluids in a thin layer]{Mathematical study of the $\beta$-plane model for rotating fluids in a thin layer}

\author{Anne-Laure Dalibard  \and Laure Saint-Raymond
}
\thanks{DMA -  UMR CNRS 8553, Ecole Normale Sup\'erieure, 45 rue d'Ulm, 75005 Paris, FRANCE. E-mail: {\tt Anne-Laure.Dalibard@ens.fr, Laure.Saint-Raymond@ens.fr}}

%\date{\today}

\maketitle

\begin{abstract}
This article is concerned with an oceanographic model describing the asymptotic behaviour of a rapidly rotating and incompressible fluid with an inhomogeneous rotation vector; the motion takes place in a thin layer. We first exhibit a stationary solution of the system which consists of an interior part and a boundary layer part. The spatial variations of the rotation vector generate strong singularities within the boundary layer, which have repercussions on the interior part of the solution.
The second part of the article is devoted to the analysis of two-dimensional and three-dimensional waves. It is shown that the thin layer effect modifies the propagation of three-dimensional Poincar\'e waves by creating small scales. Using tools of semi-classical analysis, we prove that the energy propagates at speeds of order one, i.e. much slower than in traditional rotating fluid models.\\

\selectlanguage{french}

\noindent{\sc Résumé.} On étudie ici le comportement asymptotique d'un fluide incompressible tournant à grande vitesse dans une couche mince, avec un vecteur rotation inhomogène; ce type de modèle apparaît en océanographie. On commence par exhiber une solution stationnaire du système, obtenue comme la somme d'un terme intérieur et d'un terme de couche limite. Les variations spatiales du vecteur rotation génèrent de fortes singularités dans la couche limite, qui se répercutent dans la partie intérieure de la solution.
Dans un second temps, on caractérise le comportement des ondes bi- et tri-dimensionnelles. L'effet de couche mince modifie la propagation des ondes de Poincaré (3D) en favorisant l'apparition de petites échelles. Grâce à une analyse de type semi-classique, on montre que la vitesse de propagation de l'énergie est d'ordre un, soit beaucoup plus faible que dans les modèles classiques de fluides tournants.

\selectlanguage{english}
\end{abstract}

\section{Introduction}

The goal of this article is to study the behaviour of a rotating, incompressible and homogeneous fluid, whose \textbf{rotation vector depends on the (horizontal) space variable}. We also assume that the motion of the fluid takes place in a \textbf{thin layer}. These two features are inspired from models of oceanic circulation, which are the main physical motivation for our study. We will explain more thoroughly the physical assumptions and scalings leading to our model in paragraph \ref{ssec:motivation}.

The mathematical framework of our analysis  is the following: consider the equation
\be\label{rescaled}\ba
\d_t u + \frac{1}{\eps } b(x_h) \wedge u + \begin{pmatrix}
                                            \nabla_h p \\ \frac{1}{\eta^2 } \d_z p
                                           \end{pmatrix} - \nu_h \Delta_h u -\nu_z \d_{zz} u=0,\\(x_h,z)\in\omega_h\times (0,1),
\ea
\ee
where the horizontal domain $\omega_h$ is either $\bT^2$ or $\bT\times \bR$. Equation \eqref{rescaled} is endowed with 
Navier conditions at the bottom of the domain
\be\label{bottom}
\d_z u_{h|z=0}=0,\quad u_{3|z=0}=0,
\ee
and we assume that there is a shear stress at the surface of the fluid, described by the boundary condition

\be
\ba\label{top}
\d_z  u_{h|z=1}(t,x_h)=\gamma \sigma(x_h),\\
u_{3|z=1}=0.
\ea
\ee
Above, $\eps,\eta,\nu_h,\nu_z, \gamma$ are positive parameters, whose relative size will be precised later on. Let us merely announce that $\eps,\eta,\nu_h,\nu_z$ are meant to be small, whereas $\gamma$ will be taken  large. We emphasize that equation \eqref{rescaled}, supplemented with \eqref{top}-\eqref{bottom}, is already in rescaled form. Hence all quantities are dimensionless. We refer to the next subsection for a derivation of this equation, and for a definition of the various parameters in terms of the physical quantities involved in the model.

Notice that the rotation is of order $\eps^{-1}$, with $\eps\ll 1$; hence we focus on the limit of \textbf{high rotation}. As we will see in paragraph \ref{ssec:motivation}, the parameter $\eta$ is the aspect ratio of the domain: assuming that $\eta\ll 1$ means that the characteristic horizontal length scale is much larger than the vertical one. In other words, the motion is set in a \textbf{thin layer}.

\vskip1mm

In this article, we are primarily interested in two topics: the computation of stationary solutions of our model, and the analysis of the local stability of these stationary solutions in the case $\omega_h=\bT\times \bR$. In particular, we will not address the full Cauchy problem here. Indeed, it can be proved that in the scaling which is the most relevant for our study, the energy estimates for the system \eqref{rescaled}-\eqref{bottom}-\eqref{top} explode in finite time. In a similar way, the stationary solution that we build has a size which becomes arbitrarily large as $\eps,\eta$ vanish. Hence the problem \eqref{rescaled}-\eqref{bottom}-\eqref{top} is highly singular.

To our knowledge, the asymptotic analysis of the system \eqref{rescaled} has not been addressed before: in the papers \cite{GSR} by I. Gallagher and the second author, and then \cite{DMS} by A. Dutrifoy, A. Majda and S. Schochet, the authors study the asymptotic behaviour of a shallow water system within a $\beta$-plane model (i.e. in the case $b(x_h)=\beta x_2$). This shallow water system can be obtained by considering the limit $\eta\to 0$ in \eqref{rescaled} (see \cite{GP}). Thus the studies of \cite{GSR,DMS} are concerned with the successive limits $\eta\to0, \eps\to 0$. In \cite{DG}, B. Desjardins and E. Grenier take into account the thin layer effect within the original Navier-Stokes system, but they assume that $b(x_h)=1+\eps x_2$; hence the penalization is constant at first order. Our goal is to study a crossed limit $(\eps,\eta)\to (0,0)$, with a rotation vector which has variations at the main order.

Let us now make precise the main novelties of our work: first, the construction of stationary solutions involves the definition of \textbf{boundary layer terms with a varying Coriolis factor $b$}. Since the size of the boundary layer is directly related to the amplitude of $b$, singularities appear at the vanishing points of $b$. These singularities in the boundary layer have repercussions on the interior part of the stationary solution, and make the construction much more involved than in the constant case. On the other hand, studying the stability of stationary solution when $\omega_h=\bT\times \bR$ amounts to describing the \textbf{waves in the $\beta$-plane model  with a thin layer effect}. We exhibit new types of behaviour for the Poincar\'e waves, for which we prove that dispersion takes place on a time scale much larger than usual: for instance, in Chapter 2 of \cite{GSR-handbook}, the group velocity associated with Poincar\'e waves (i.e. the speed at which energy propagates) is of order $\eps^{-1}$, while the group velocity in the present setting is of order one. The proof of this fact uses tools of semi-classical analysis, in the spirit of the recent papers by C. Cheverry, I. Gallagher, T. Paul and the second author (see \cite{CGPSR,CGPSR-papier}). Notice also that the presence of dispersion in an oceanographic model is itself unusual: indeed, most models are set in a compact domain (see \cite{CDGG}), where no dispersion can occur. Moreover, the most commonly used whole-space model is the shallow water system within the $\beta$-plane model (see \cite{GSR,DMS}), for which waves are trapped into a waveguide, and thus no dispersion occurs either.

\vskip1mm

In the next paragraphs, we explain which physical assumptions led to the system \eqref{rescaled}. We then present our main results. Eventually, let us point out that the structure of the stationary solution which will be built in this article enforces particular shapes for the isothermal surfaces inside the fluid (the so-called ``thermocline''). We present a few results in this regard in paragraph \ref{ssec:thermocline}.

\subsection{Physical derivation of equation \eqref{rescaled}}\label{ssec:motivation}$ $

Let us now explain in which regime oceanic currents can be modeled by equation \eqref{rescaled}. In this subsection, we denote by $u$ the velocity of oceanic currents in \textit{dimensional variables.} The dimensionless variables, i.e. the ones in which equation \eqref{rescaled} is written, will be denoted with a prime.

$\bullet$ As a starting point, we recall that the ocean can be considered as an incompressible fluid with variable density $ \rho$. In order to simplify the analysis, we neglect the variations of density, which are of order $10^{-3}$ in the ocean. Consequently, the velocity $u$ satisfies the Navier-Stokes equations, with a Coriolis term accounting for the rotation of the Earth

\begin{equation}
\label{NS-ocean}
\begin{aligned}
\rho_0 \left[\d_t u+(u\cdot \nabla )u \right]+\nabla p =\cF +  \rho_0 u\wedge\Omega \,,\\
\nabla \cdot u =0\,,
\end{aligned}
\end{equation}
where $\cF$  denotes the  frictional
force acting on the fluid, $\Omega$ is the (vertical component of the) Earth rotation vector,  $p$ is the pressure defined as the Lagrange multiplier associated with the incompressibility constraint, and $\rho_0$ is the (constant) value of the density.

Since we have chosen to work on large horizontal scales (see below), equation \eqref{NS-ocean} should be written in spherical coordinates. However, computations involving spherical coordinates are much lengthier, and do not change substantially the physical phenomena we wish to highlight, at least at a formal level (see \cite{Pedlosky1}). Thus in the rest of the article, we neglect the \textbf{curvature of the Earth} (but we keep a varying Coriolis factor nonetheless). Note also that we neglect the influence of the horizontal component of the Earth rotation vector, which is classical in an oceanographic framework (see \cite{GSR-handbook}).

The observed persistence over several days of large-scale waves  in the oceans shows that {\bf frictional forces} $\cF$ are weak, almost  everywhere, when
compared with the Coriolis acceleration and the pressure gradient, but large when compared with the kinematic viscous dissipation of water. 
One common but not very precise notion is that  small-scale
motions, which appear sporadic or on longer time scales, act to  smooth and mix
properties on the larger scales by processes analogous to molecular,  diffusive
transports.
For the present purposes it is only necessary to note that one way to  estimate
the dissipative influence of smaller-scale motions is to retain the same
representation of the frictional force 
$$\cF= A_h \Delta_h u+A_z \d_{zz} u$$
where $A_z$ and $A_h$ are respectively the vertical and horizontal turbulent viscosities, of much larger magnitude than the
molecular value, supposedly because of the greater efficiency of  momentum
transport by macroscopic chunks of fluid. 
Notice that $A_z\neq A_h$  is therefore natural in a geophysical framework (see \cite{Pedlosky1}). Moreover, models of oceanic circulation usually assume that the vertical viscosity $A_z$ is not constant (see \cite{BD,PP}); we choose to retain only the mean boundary value of the vertical viscosity $A_z$, since one of the motivations for our work was to compute the boundary layer terms in a context where $\Omega$ is not constant.

$\bullet$
Let us now describe the boundary conditions associated with \eqref{NS-ocean}: typically, Dirichlet boundary conditions are enforced at the bottom of the ocean and on the lateral boundaries of the horizontal domain $\omega_h$ (the coasts), i.e.
\be
\begin{aligned}
u_{|z=h_B(x_h)}=0 \quad\text{(bottom)},\\
u_{|x\in \d \omega_h}=0\quad\text{(coasts).}
\end{aligned}
\ee
In equation \eqref{rescaled}, we have  {\bf neglected the  effects of the lateral boundary conditions} by considering the case when $\omega_h$ is either $\bT\times\bR$ or $\bT^2$.  By doing so, we have deliberately prohibited the apparition of strong western boundary currents, which play a crucial role in the oceanic circulation (e.g. the Gulf Stream, the Kuroshio current). These horizontal boundary layers are believed to be responsible for the vertical structure of the ocean, and for the creation of large eddies. In the linear case, the mathematical treatment of these layers, called \textit{Munk layers}, is performed by B. Desjardins and E. Grenier in \cite{DG}. Their study could probably be mimicked in the present paper without strong modifications; however, we have chosen to leave this issue aside in order to focus on the other features of the model. Note that in the nonlinear case, the analysis of lateral boundary layers is completely open from a mathematical point of view.

In a similar fashion, for the sake of simplicity, we did not take into account  the {\bf topography of the bottom} in \eqref{bottom} (i.e. we have taken $h_B\equiv 0$), and we took Navier instead of Dirichlet boundary conditions, meaning that oceanic currents achieve perfect slip on the bottom. This choice simplifies the mathematical analysis, since it avoids the apparition of Ekman boundary layers on the lower boundary. The treatment of Ekman boundary layers in the case of a Dirichlet boundary condition with $h_B\equiv 0$ is in fact completely similar to the one of Ekman boundary layers due to the wind at the surface of the fluid, which is performed in section \ref{sec:BL}. Hence changing Dirichlet into Navier boundary conditions is not a strong mathematical restriction. The case of Ekman boundary layers with a non-zero $h_B$ has been addressed by B. Desjardins and E. Grenier  \cite{DG}, N. Masmoudi \cite{M}, and D. G\'erard-Varet \cite{dgv} in the case of a constant $b$, when $h_B$ is of the order of the Ekman boundary layer (see below). In the present case, if the same assumption is satisfied, it can be checked that the case of a non-constant $h_B$ can be treated with the same arguments as the ones in section \ref{sec:BL}.

We assume that the upper surface, which we denote by $\Gamma_s$, has an equation of the type $z=h_S(t,x_h)$.
As boundary conditions on $\Gamma_s$, we enforce (see \cite{GP})
\begin{equation}
\label{top-bis}\begin{aligned}
\Sigma\cdot n_{\Gamma_s}=\sigma_w,\\
\frac{\d}{\d t}\mathbf 1_{0\leq z \leq  h_S(t,x)} + \mathrm{div}_x(\mathbf 1_{0\leq z \leq  h_S(t,x_h)} u)=0\end{aligned}
\end{equation}
where $\Sigma$ is the total stress tensor of the fluid, and $\sigma_w$ is a given stress tensor describing the {\bf wind on the surface of the ocean}. In general, $\Gamma_s$ is a free surface, and a moving interface between air and water, which has its own self consistent motion. In \eqref{top}, we have assumed that
$$
h_S(t,x_h)\equiv D,
$$
where $D$ is the typical depth of the ocean. Hence  (\ref{top}) is a rigid lid approximation, which is a drastic, but standard simplification.
The justification of (\ref{top}) starting from a free surface is mainly open from a mathematical point of view;  we refer to \cite{AP} for the derivation of Navier-type wall laws for the Laplace equation, under general assumptions on the interface, and to \cite{LT} for some elements of justification in the case of the great lake equations.  Nevertheless, from a physical point of view, the simplification does not appear so dramatic, since in any case the free surface is so turbulent with waves and foam, that only modelization is tractable and meaningful. Condition (\ref{top}) is a simple modelization which already catches most of the physical phenomena (see \cite{Pedlosky1}).

\noindent $\bullet$ Let us now evaluate the order of magnitude of the different parameters occurring in \eqref{NS-ocean}, and write the equations in a dimensionless form. 
We set
$$\begin{aligned}
   u_h= U u_h',\quad
u_3=W u_3',\\
x_h=H x_h',\quad
z=Dz',\\
  \end{aligned}
$$
where $U$ (resp. $W$) is the typical value of the horizontal (resp. vertical) velocity, $H$ is the horizontal length scale, and $D$ the depth of the ocean. In order that $u'(x')$ remains  divergence-free, we choose
$$
W= \frac{U D}{H}.$$
A typical value of the horizontal velocity  for  the mesoscale eddies that have been observed in western Atlantic (see for instance \cite{Pedlosky1}) is $U\sim  1 \ \mathrm{cm\cdot s^{-1}}$. Moreover, the typical horizontal and vertical scales which we are interested in are 
$$
H\sim 10^4\ \mathrm{km}, \quad\text{and } D\sim 4\  \mathrm{km}.
$$
Notice that we work on an almost planetary scale, which justifies the use of a varying rotation vector. Concerning the rotation, we write $\Omega=\Omega_0\sin(\theta)$, where $\theta$ is the latitude, and $\Omega_0=2\pi /\mathrm{day}\sim 7 \cdot 10^{-5} s^{-1}.$ Eventually, we consider the motion on a typical time scale $T$, with $T$ of the order of a few months ($T\sim 10^7 s$).
With these values, we get
$$
\eps:= \frac{1}{ T \Omega_0 }\sim  10^{-3} ,
$$
and hence $\eps\ll 1$ (notice that the parameter $\eps$ is dimensionless). Thus the asymptotic of fast rotation (small Rossby number) is valid.

\vskip1mm
Thus the dimensionless system (see for instance \cite{Pedlosky1,gill}) becomes
\be\label{NS}
\begin{aligned}
\d_t u' + \frac{TU}{H} u'\cdot \nabla u' + \frac{1}{\eps}b(x_h) e_3\wedge u' + \begin{pmatrix} \nabla_h p' \\ \frac{1}{\eta^2}\d_z p'
                                                        \end{pmatrix}
-\nu_h\Delta_h u' -\nu_z \d_{zz} u'=0,\\
\nabla\cdot u'=0,
\end{aligned}
\ee
where $\eta:=D/H\sim  4 \cdot 10^{-4}$ is the aspect ratio, and the vertical and horizontal viscosities are defined by 
$$
\nu_z:=\frac{T A_z}{\rho_0  D^2}, \quad\nu_h=\frac{A_h T}{\rho_0  H^2}.
$$
Typical values for the  turbulent viscosities are (see \cite{gill}) $A_z/\rho_0\sim 10^{-4}-10^{-3}\;\mathrm{m^2\cdot s^{-1}}$, and $A_h/\rho_0\sim 10^{4}-10^{5}\;\mathrm{m^2\cdot s^{-1}}$, which yields in the present case $\nu_z\sim 10^{-3}$ and $\nu_h\sim 10^{-10}-10^{-9}$.

The boundary conditions are \eqref{top}, \eqref{bottom}, with
$$
\gamma:= \frac{|\sigma_w| D }{A_z U}.
$$
Notice that with the time scale chosen above, the convective term is of order $10^{-2}\ll 1$; hence we neglect it in the rest of the study. Note however that the effect of this term is expected to be large if the waves associated with \eqref{NS} are resonant, and small if they are dispersive. Thus the rigorous treatment of the convective term requires a mathematical analysis which goes beyond the scope of this article, and which we deliberately leave aside from now on.

\vskip2mm
In the rest of the article, the relative size of the parameters will be chosen as follows: the most important feature of our analysis is that $\eta$ and $\eps$ are chosen of the same order. In order to keep the number of different small parameters to a minimum, we also choose to take $\nu_z=\eps$, and $\gamma=\eps^{-2}$; with this last choice, the interior part of the stationary solution built in the next sections will be of order one. Concerning the size of $\nu_h$, our analysis  allows us to consider horizontal viscosities $\nu_h=o(\eps)$, which is compatible with the orders of magnitude given above.

\subsection{Main results}\label{ssec:results}~

We present here two types of results: first, we build an approximate ``stationary'' solution of the system \eqref{rescaled}, endowed with the boundary conditions \eqref{top}-\eqref{bottom}. The problem studied is rather different from the Cauchy problem, since no initial data is prescribed. The goal is merely to compute a solution of \eqref{rescaled}, and to investigate its asymptotic behaviour as $\eps$ vanishes.

Once the behaviour of the stationary solution is understood, we study its local stability; since equation \eqref{rescaled} is linear, this is equivalent to studying the Cauchy problem for equation \eqref{rescaled}, with homogeneous Navier conditions at $z=0$ and $z=1$. We then exhibit Rossby waves, which are essentially two-dimensional, and Poincar\'e waves, which are fluctuations around the three dimensional part of the initial data, and which take place on a much larger time scale.

Let us now state our result about stationary solutions: since the vertical viscosity is small (we take $\nu_z=\eps\ll 1$), it disappears from the asymptotic system. As a consequence, solutions of the limit system cannot satisfy the boundary conditions. Thus boundary layer terms are introduced, which restore the correct boundary conditions. Hence the  stationary  solution built here is composed of an interior part and a boundary layer part.

We state our result in the case $\omega_h=\bT\times \bR$, and explain below the Theorem the main differences when $\omega_h=\bT^2$. Throughout the paper, we set$$\omega:=\omega_h\times (0,1).$$

\begin{Thm}[Stationary solutions of \eqref{NS}] Let $\omega_h=\bT\times \bR$. 

Assume that $\nu_h=o(\eps)$ and that $\eta=\nu_z=\eps,$ $\gamma=\eps^{-2}$.

Let $\sigma \in H^2(\omega_h)\cap W^{2,\infty}(\omega_h)$ such that
\be\label{hyp:sigma_quadra}
\begin{aligned}
   \left| \sigma(x,y) \right|,\;    \left|\d_x \sigma(x,y) \right|\leq C y^2\quad\forall (x,y)\in\omega_h,\\
\left| \d_y \sigma(x,y)  \right| \leq C |y|\quad\forall (x,y)\in\omega_h
  \end{aligned}
\ee
and such that the following compatibility condition is satisfied
\be
 \int_\bT \sigma_1(x, y)\:dx=0\quad\forall y.\label{compatibility}
\ee

Assume that the Coriolis factor $b$ satisfies the following assumptions:
\be\begin{aligned}
b(x,y)=b(y)\quad \forall (x,y)\in\omega_h,\quad \text{with }b\in W^{2,\infty}_{\text{loc}}(\bR),\\
b(y)\neq 0 \text{ for }y\neq  0,\text{ and }\exists C>0,\  |b(y)|\geq C \text{ for }|y|\geq 1,\\
\exists c>0,\ c^{-1}\leq b'(y)\leq c\ \forall y,\quad b(y)\sim \beta y \text{ for }y\to 0. 
  \end{aligned}
\label{hyp:Coriolis}\ee

Then there exists stationary functions $(\ustat, \pstat)\in L^2(\omega)\cap H^1(\omega)$, such that $\ustat$ satisfies  \eqref{top}, \eqref{bottom} and
$$\begin{aligned}
  \frac{1}{\eps} b(y)(\ustat_h)^\bot +  \nabla_h \pstat - \eps \d_{zz} \ustat_h -\nu_h \Delta_h\ustat_h =r_h^1 + r_h^2\\
\frac{1}{\eps^2} \d_z\pstat - \eps \d_{zz} \ustat_3 -\nu_h \Delta_h\ustat_3 =r_3^1 + r_3^2,
  \end{aligned}
$$with
$$\begin{aligned}
r_h^1=o(1)\text{ in }L^2(\omega),\quad r_h^2=o(\sqrt{\nu_h})\text{ in }L^2([0,1], H^{-1}(\omega_h)),\\
r_3^1=o(\eps^{-1})\text{ in }L^2(\omega),\quad r_3^2=o(\eps^{-1} \sqrt{\nu_h})\text{ in }L^2([0,1], H^{-1}(\omega_h)).
  \end{aligned}
$$

Moreover, $\ustat$ can be decomposed as 
$$
\ustat= \ubl + \uint,
$$
where $\ubl$ is a term located in a boundary layer of size $\eps$, in the vicinity of the surface, and $\uint$ is an interior term. The functions $\ubl $ and $\uint$ satisfy the following estimates
\be
\begin{aligned}
\| \uint\|_{L^2(\omega)}\leq C \| \sigma\|_{H^2(\omega_h)},\\ 
\|\ubl_h \|_{L^2(\omega)}\leq \frac{C}{\sqrt{\eps}} \| \sigma\|_{H^1(\omega_h)},\\
\| \ubl_3\|_{L^2(\omega)}\leq C \| \sigma\|_{H^1(\omega_h)}.
\end{aligned}
\ee
\label{thm:stat}
\end{Thm}

If $\omega_h=\bT^2$, the result remains true under slightly different conditions on $\sigma$ and $b$. More precisely, we assume that $\sigma\in H^2(\bT^2)$ satisfies \eqref{hyp:sigma_quadra}, \eqref{compatibility}, and that 
$$
d(\Supp \sigma, (\bT,1/2))>0
$$
In other words, $\sigma$ vanishes in a neighbourhood of $(x,1/2)$ for all $x\in\bT$ (and by periodicity, in a neighbourhood of $(x,-1/2)$ also).

We assume furthermore that $b(x,y)=b(y)$ with
\be\begin{aligned}
b\in L^\infty(\bT)\text{ and }b\in W^{2,\infty}(K)\ \forall K\subset \bT\text{ compact s.t. }d(K,1/2)>0,\\
b(y)\neq 0 \text{ for }y\neq  0,\text{ and }\exists C>0,\  |b(y)|\geq C \text{ for }|y|\geq 1/4,\\
 b(y)\sim \beta y \text{ for }y\to 0,\\
\forall K\subset \bT\text{ compact s.t. }d(K,1/2)>0,\ \exists c_K>0,\ c_K^{-1}\leq b'(y)\leq c_K\ \forall y\in c_K. 
  \end{aligned}
\label{hyp:Coriolis-bis}\ee

In other words, we do not assume that $b\in W^{2,\infty}(\bT)$: $b$ may have a discontinuity at  $y=1/2  $. But we require that $\sigma $ vanishes in a neighbourhood of that singularity, so that all terms of the type $\sigma b$, $\sigma/b$, $\sigma/b'$ are well-defined and $\bT^2$-periodic.

\begin{Rmk}
\begin{enumerate}[\bf (i)]

\item The assumptions \eqref{hyp:Coriolis}-\eqref{hyp:Coriolis-bis} on the Coriolis factor $b$ is satisfied in two particular cases: 
\begin{itemize}
\item $b(y)=\beta y$, with $\omega_h=\bT\times \bR$: this approximation is particularly relevant for the motion of equatorial currents, and is used in particular in \cite{GSR}, \cite{DMS}. 

	\item $b(y)=\sin(\pi y/2)$, with $\omega_h=\bT^2$: this is the case of a real ocean, whose study takes place on a planetary scale. Of course, in this case, the effect of the curvature of the Earth should be taken into account, which we have chosen not to do here (see the discussion in the previous paragraph). 

\end{itemize}

 \item Notice that in the above Theorem, it is assumed that the surface stress vanishes near $y=0$. Although this assumption stems from mathematical considerations, it is in fact quite reasonable in an oceanographic context. Indeed, it is a well-known phenomena that there are no steady surface winds near the equator: as trade winds coming from the North and South meet, they are heated and produce upward winds. The area of calm in the vicinity of the equator is called the \textit{Doldrums}.

\item The compatibility condition \eqref{compatibility} means that there is no zonal average wind. This condition is of course not realistic from a physical point of view, but it is the price to pay for working with a domain with no boundary in $x$. If the horizontal domain $\omega_h$ is replaced by $[0,1]\times \bT$ or $[0,1]\times \bR$, this condition disappears; the (mathematical) counterpart lies in the construction of the horizontal boundary layer terms, the so-called Munk layers discussed in the previous paragraph.

\item In general, the size of the boundary layer term $\ubl$ is much larger than that of the interior term. This means that the greatest part of the energy is concentrated in a boundary layer located in the vicinity of the surface. In  the original variables, it can be checked that the boundary layer carries an energy of order $\rho_0 U^2 H^3$, while the energy contained in the interior of the domain is of order $\rho_0 U^2 H^2 D$.

This is in fact a consequence on the requirements on $\ubl$, $\uint$, and not an artefact of our model. Indeed, assume that the functions $\ubl,\uint$ are such that 
$$\begin{aligned}
  \|\uint_{3|z=1} \|_{L^2(\omega_h)}\sim \|\uint_{h|z=1} \|_{L^2(\omega_h)}\sim \|\uint_{h} \|_{L^2(\om)},\\
\uint_{3|z=1} =- \ubl_{3|z=1},
  \end{aligned}
$$
and assume that $\ubl,\uint$ are divergence free and that $\ubl$ is located in a boundary layer of size $\delta_{E}$ (where $E$ stands for `Ekman') near the surface. Denote by $A^{BL}_{h}, A^{BL}_3$ the size of $\ubl_h,\ubl_3$ in $L^\infty$, and by $A^{int}$ the size of $\uint$ in $L^2(\omega).$

The assumptions above entail that $A^{BL}_3=A^{int}$; on the other hand, since $\ubl$ is divergence free, we have
$$A^{BL}_h=\frac{1}{\delta_E} A^{BL}_3=\frac{1}{\delta_E}A^{int}.$$

Consequently, since 
$$u^{BL}_h\sim A^{BL}_h\exp\left(-\frac{1-z}{\delta_E}  \right),$$
we infer that
$$
\| \ubl_h\|_{L^2(\omega)}=\sqrt{\delta_E} A^{BL}_h = \frac{1}{\sqrt{\delta_E}}A^{int}.
$$
Thus the energy in the boundary layer is always larger than the energy in the interior of the fluid with this type of model. The assumption that $\|\uint_{3|z=1} \|\sim\|\uint_{h|z=1}\|$ stems from observations of the isothermal surfaces in the ocean, as we will explain in the next paragraph. From a physical point of view, having $\|\ubl\|$ much larger than $\|\uint\|$ is in fact quite reasonable: indeed, it is observed that subsurface currents generally travel at a much slower speed when compared to surface flows.

\end{enumerate}
 
\label{rem:stat_sol}

\end{Rmk}
 
Let us also emphasize that in the case of the $f$-plane model (i.e. when the   rotation vector $b$ is constant), the result of Theorem \ref{thm:stat} is false in general. Indeed, the interior part of the solution must satisfy the geostrophic system, namely
$$
\ba
u_h^\bot + \nabla_h p=0,\\
\Div_h u_h + \d_z u_3=0,\\
\d_z p=0,
\ea
$$
and thus $u$ is a two-dimensional divergence free vector field. In other words, $u_3\equiv 0$  and thus the Ekman pumping velocities must be zero at first order. Consequently, the  interior part of the solution cannot be wind-driven at first order. 

\vskip1mm

We now address the question of the stability of the stationary solution constructed above:
\begin{Thm}[Waves associated with equation \eqref{NS}]
Assume that $\omega_h=\bT\times \bR$, and that $b(x_h)=\beta y$ for all $x_h=(x,y)\in\omega_h.$

For any $\eps>0$, let $v^\eps$ be a solution to the propagation equation
$$\ba
\d_t u + \frac{1}{\eps } b(x_h) \wedge u + \begin{pmatrix}
                                            \nabla_h p \\ \frac{1}{\eps^2 } \d_z p
                                           \end{pmatrix} - \nu_h \Delta_h u -\eps \d_{zz} u=0,\\(x_h,z)\in\omega_h\times (0,1),
\ea
$$
with $\nu_h= O(\eps^2)$, supplemented with homogeneous boundary conditions
$$
\d_z u_{h|z=0}=\d_z u_{h|z=1}=0,\quad u_{3|z=0}=u_{3|z=1}=0.
$$
Then $v^\eps$ can be decomposed as the sum of 
\begin{itemize}
	\item a stationary part $\bar v^\eps(t,y) = \int_{\bT}\int_0^1 v^\eps(t,x,y,z) \: dx\:d z$, which satisfies
$$\d_t \bar v^\eps-\nu_h \d_y^2 \bar v^\eps=0, $$
\item{\bf Rossby waves} $v^\eps_R=\int v^\eps dx_3 - \bar v^\eps$ corresponding to the 2D vorticity propagation
$$
\d_t \zeta^\eps_R+ \frac{\beta}{\eps}\d_x \Delta_h^{-1} \zeta^\eps_R-\nu_h \Delta_h \zeta^\eps_R =0,
$$
where $\zeta_R^\eps=\rot_h v^\eps_R$, 
\item and {\bf gravity waves} $v^\eps_G=v^\eps-\int v^\eps dx_3$.
\end{itemize}

Rossby and Gravity waves have a dispersive behaviour as $\eps$ vanishes:
\begin{itemize}
\item Rossby waves disperse on a small time scale
$$\forall t>0 ,\quad \forall K\subset\subset  \omega,\quad \| v_R^\eps(t)\|_{L^2(K)} \to 0 \hbox{ as } \eps \to 0\,,$$
since we have assumed that $y\in \bR$;
\item Gravity waves generate fast oscillations with respect to $y$, which slows down the propagation
$$ \forall K\subset\subset  \omega,\quad \| v_G^\eps(t)\|_{L^2(K)} \to 0 \hbox{ as } (\eps,t) \to (0,\infty)\,.$$
\end{itemize}
\label{thm:propagation}
\end{Thm}

\begin{Rmk}$ $
\begin{itemize}
	\item Notice that the energy associated with gravity (or Poincar\'e) waves propagates on a time scale much larger than the one of Rossby waves. This is due to the thin layer effect, which causes the apparition of small scales in the variable $y$. 
\item 
The field $\bar v^\eps$ is said to be ``stationary'' because the horizontal viscosity $\nu_h$ is small: hence
$$
\bar v^\eps(t)\approx \bar v^\eps_{|t=0}\quad \text{in }L^2 $$
on time scales of order one.
\end{itemize}

\end{Rmk}

\begin{Cor}

Assume that $\omega_h=\bT\times \bR$, and that $b(x_h)=\beta y$ for all $x_h=(x,y)\in\omega_h.$ Assume that $\nu_h=O(\eps^2).$

 For any $\eps>0$, let $u^\eps$ be a solution of \eqref{rescaled} supplemented with \eqref{bottom}-\eqref{top}, and assume that
$$
\sup_{\eps>0} \| u^\eps_{h|t=0}- \ustat_h \|_{L^2(\omega)} + \eps \| u^\eps _{3|t=0}- \ustat_3 \|_{L^2(\omega)}<+\infty.
$$
Then for any finite time $t>0$, 
$$
u^\eps(t)- \ustat \sim \bar v^\eps(t) + v_G^\eps(t) \hbox{ in } L^2_{loc}(\omega)
$$
where $v_G^\eps$ is the (slow propagating and fast oscillating) gravity part of the velocity field $v^\eps$ defined in Theorem \ref{thm:propagation}, and $\bar v^\eps$ is the stationary part of $v^\eps$. 

(Note in particular that the vertical component of the velocity $u_\eps$ is not expected to be bounded - as is usually claimed for shallow water approximation.)

\end{Cor}

The above Corollary is an immediate consequence of Theorems \ref{thm:stat} and \ref{thm:propagation}, together with the energy inequality.

\subsection{Towards a mathematical derivation of the thermocline}\label{ssec:thermocline}$ $

In this paragraph, we try to justify the shape of the surfaces of equal temperature in the ocean, in view of the results of Theorem \ref{thm:stat}. 

The isothermal surface which is located just below the Ekman boundary layer is of special interest to oceanographers, due to its importance on the global oceanic circulation (see \cite{Pedlosky1,Pedlosky2,LPS}).
Figure \ref{fig:thermocline} below shows the longitudinal variations of the temperature in the Pacific ocean in a layer of $1000\;\mathrm{m}$ depth below the surface.
\begin{figure} [h] %  figure placement: here, top, bottom, or page
   \centering
   \includegraphics{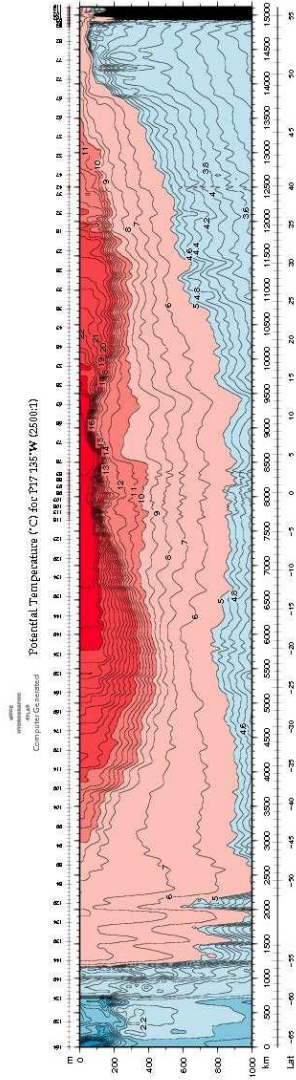} 

\caption{Longitudinal section of the surfaces of equal temperature in the Pacific ocean \textit{(from the WOCE Pacific Ocean Atlas).}}

\label{fig:thermocline}
\end{figure}
In particular, there are zones in which the temperature surfaces are tilted up (that is, there is a flux of cold water towards the surface); this phenomenon cannot always be accounted for by the heating differences at the surface, as shows the upward flux of cold water in the equatorial zone. The physical justification of these particular shapes is the following: inside the ocean, the temperature $T$ solves an equation of the kind
$$
u^*\cdot \nabla T -\kappa \Delta T=0,
$$
where $u^*$ is the velocity of oceanic currents (in the dimensional variables) and $\kappa$ is the heat conductivity coefficient. If the temperature diffusion can be neglected, this equation takes the form $$ u^* \cdot \nabla T=0,$$ which means that $u$ is a tangent vector to the isothermal surfaces. Consequently, the temperature surfaces are tilted up (or down) if and only if $u^*_{3|\text{surface}}\neq 0$, or more precisely, if $|u_{3|z=1}|/|u_{h|z=1}|=O(1)$ in rescaled variables. This justifies the assumption $$\|\uint_{3|z=1}\|\sim \|\uint_{h|z=1}\| $$ in the previous paragraph (see Remark \ref{rem:stat_sol} \textbf{(iv)}).

In that regard, the special solution constructed in Theorem \ref{thm:stat} is of particular interest. Indeed, in rescaled variables, we have (see section \ref{sec:int})
$$
\uint_{3|z=1}=-\frac{1}{b}\rot_h\sigma - \frac{b'}{b^2}\sigma_1.
$$
Hence $u_{3|\text{surface}}\neq 0,$ and our model predicts that the temperature surfaces are indeed modified by the Ekman pumping velocity.

We now give a rigorous result about the asymptotic shape of the temperature in our model. 
We denote with a star the original variables. We write 
$$
T(x_h^*, z^*)=T_0 + T_1 \theta\left(\frac{x_h^*}{H}, \frac{z^* }{D} \right),
$$
with the same notations as in paragraph \ref{ssec:motivation}. The temperature $T_0$ is a reference temperature (for instance, $T_0=10^\circ\mathrm{C}$), whereas $T_1$ is the order of magnitude of the variations of the temperature. Performing the same change of variables as in paragraph \ref{ssec:motivation}, we obtain
$$
u\cdot \nabla \theta - \lambda \eta^2 \Delta_h \theta - \lambda\d_{zz}\theta=0,
$$
where the diffusion coefficient $\lambda$ is given by
$$
\lambda=\frac{\kappa L}{D^2 U}.
$$
We recall that $\eta$ is the aspect ratio of the domain; as in Theorems \ref{thm:stat} and \ref{thm:propagation}, we take $\eta=\eps.$  With the notation of Theorem \ref{thm:stat}, our result is the following:

\begin{Prop}Let $\lambda>0.$
Assume that the wind stress $\sigma\in H^s(\omega_h)$ is such that 
\be
\begin{aligned}
   \left| \sigma(x,y) \right|\leq C y^k\quad\forall (x,y)\in\omega_h,\\
\left| \nabla \sigma(x,y)  \right| \leq C |y|^{k-1}\quad\forall (x,y)\in\omega_h,
  \end{aligned}
\ee
for some $k,s\geq 2$ chosen sufficiently large, and assume that 
\be\label{hyp:u_h_lambda}\| \nabla_h\uint_h\|_{L^\infty(\omega)} \leq \frac{\lambda}{4}.\ee
Let $\theta$ be the solution of the equation
\be\label{eq:theta}
\ustat \cdot \nabla \theta - \eta^2\lambda\Delta_h \theta - \lambda \d_{zz} \theta=0,
\ee
supplemented with the boundary conditions
\be\label{BC-theta}
\ba
\theta_{|z=1}=\theta_1,\quad\d_z \theta_{|z=0}=0,
\ea
\ee
for some function $\theta_1\in H^2(\omega_h)$.

Define the function $\theta^{app}$ by
$$
\theta^{app}(x_h,z)= \bar \theta (x_h,z) + \eps \theta^{BL}\left(x_h,\frac{1-z}{\eps}  \right),
$$
where $\bar \theta$, $\theta^{BL}$ are solutions of
\be\label{eq:btheta}
\ba
-\lambda \d_{zz} \bar \theta + \uint\cdot \nabla \bar \theta =0 \quad\text{in }\omega,\\
\bar \theta_{|z=1}=\theta_1,\quad\d_z \bar \theta_{|z=0}=0,
\ea
\ee
and
$$
\ba
-\lambda \d_{\zeta\zeta} \theta^{BL}(x_h,\zeta) + \eps \ubl_h(x_h,1-\eps\zeta)\cdot \nabla_h \theta_1=0,\\
\theta^{BL}(x_h,\zeta)\underset{\zeta\to\infty}{\longrightarrow} 0.
\ea
$$

Then as $\eps\to 0$,
$$
\| \theta - \theta^{app}\|_{L^2(\omega)} + \| \d_z (\theta - \theta^{app})\|_{L^2(\omega)} \to 0.
$$
\label{prop:thermocline}
\end{Prop}
\begin{Rmk}

\begin{enumerate}[\bf (i)]
	\item The assumption \eqref{hyp:u_h_lambda} on the size of $\nabla_h \uint_h$ is purely technical, and does not have any physical interpretation. It rises from the fact that equation \eqref{eq:btheta} on $\bt$ is degenerate in the horizontal variable; we refer to section \ref{sec:thermocline} for more details. We emphasize in particular that if \eqref{hyp:u_h_lambda} is not satisfied, equation \eqref{eq:btheta} is still well-posed in $L^2(\omega_h, H^1(0,1))$; however, in this case, we are no longer able to prove the convergence.

\item The boundary conditions \eqref{BC-theta} mean that the atmosphere acts like a thermostat for the ocean, and that there is no heat flux at the bottom of the ocean. Both assumptions seem reasonable from a physical point of view, although other boundary conditions might also make sense: for instance, it could also be assumed that  the heat flux at the surface is a given function of the latitude.

\item Let us mention a last direction towards which the physical accuracy of our model could be improved. When considering the spatial variations of the temperature, it would be more reasonable to consider a model which couples the velocity of ocean currents and the temperature, in the spirit of \cite{CT}. However, the relevant scalings within such models are not completely clear. Furthermore, the analysis in Chapter 6 of \cite{Pedlosky1} shows that for such problems, the curvature of the Earth should be taken into account.  
Hence we leave this issue aside in the present paper.

\end{enumerate}

\label{rem:thermocline} 
\end{Rmk}

\vskip1mm

The construction of the article is as follows: in the next two sections, we construct the stationary solution of equation \eqref{rescaled}, starting with the boundary layer part, and then building the interior part by solving the geostrophic equations with a Dirichlet boundary condition on the vertical component. Then, we prove Theorem \ref{thm:propagation} in sections \ref{sec:2D} and \ref{sec:3D}, by treating separately the two-dimensional and three-dimensional parts of the initial data. Eventually, section \ref{sec:thermocline} is dedicated to the proof of Proposition \ref{prop:thermocline}.

\section{The boundary layer part of the stationary solution}

\label{sec:BL}

In this section, we construct functions $\ubl,\pbl$ which are approximate stationary solutions of equation \eqref{rescaled} (in the sense of Theorem \ref{thm:stat}), and which satisfy the horizontal part of the boundary condition \eqref{top}. These functions are located in a boundary layer in the vicinity of the surface $z=1$. Our methodology is the following: we first assume that $\nu_h=0$, and we use the classical construction of Ekman layers in this case. We then derive several estimates on the functions thus obtained. Eventually, we estimate the error terms in equation \eqref{rescaled} which are due to the fact that $\nu_h$ is non zero.

\subsection{Construction in the case $\nu_h=0$}$ $

When the horizontal  viscosity vanishes, the construction of the boundary layer is exactly the same as in the $f$-plane model, i.e. when the function $b$ does not depend on $x_h$. Indeed, in this case the variable $x_h$ is merely a parameter of the equation, and building the boundary layer term amounts to solving an equation on the rate of exponential decay. For more results regarding classical boundary layers, we refer to \cite{CDGG,M,MR,Ro}. Nonetheless, let us stress that even though the construction itself is the same, the estimates become much more involved than in the case of the $f$-plane model. Indeed, the vanishing points of $b$ create singularities, and prevent the boundary layer terms to be in $L^2$ in general. Hence, assumptions on the stress $\sigma$ have to be introduced in order to handle these singularities.

The construction of the boundary layer term is as follows: we wish to construct an approximate solution $(\ubl,\pbl)$ of \eqref{rescaled}, such that \eqref{top} is satisfied. Furthermore, we assume that this approximate solution is small outside a boundary layer located in the vicinity of the surface $z=1$. Hence, we look for $\ubl,\pbl$ in the form
$$\begin{aligned}
  \ubl_h(t,x_h,z)=\Ubl_h\left( x_h,\frac{1-z}{\eps} \right),\\
\pbl(t,x_h,z)=\Pbl\left( x_h,\frac{1-z}{\eps} \right).
  \end{aligned}
$$
We assume that $\Ubl,\Pbl$ together with all their derivatives vanish as $\zeta\to\infty$, where $\zeta$ stands for the  rescaled variable $(1-z)/\eps$. Inserting the above Ansatz into equation \eqref{rescaled} yields
\be\label{eq:BL}
\left\{
\begin{array}{l}
b(x_h)(\Ubl_h)^\bot - \d_\zeta^2\Ubl_h + \eps \nabla_h \Pbl=0,\\
-\d_\zeta^2 \Ubl_3 - \frac{1}{\eps^2}\d_\zeta \Pbl=0,\\
\Div_h \Ubl_h - \frac{1}{\eps}\d_\zeta U_3=0.\end{array}
 \right.
\ee
The last two equations entail that
$$
\Pbl=-\eps^2\d_\zeta \Ubl_3 =-\eps^3 \Div_h \Ubl_h.
$$
We henceforth neglect the pressure term in the equation on $\Ubl_h$. Then, we set, as usual (see for instance \cite{M}), 
$$
U^\pm_h:= U_h\pm i U_h^\bot.
$$
Above and in the rest of the article, for all $u=(u_1,u_2)\in\bR^2$, $u^\bot:=(-u_2,u_1)$.

An easy calculation leads to
$$\begin{aligned}
   -\d_\zeta^2 U_h^\pm \mp i b U_h^\pm=0,\\
\d_\zeta U^\pm_{h|\zeta=0}=-\frac{1}{\eps}(\sigma \pm i \sigma^\bot).
  \end{aligned}
$$
Consequently, $U_h^\pm$ is an exponentially decaying function of the form
$$
U_h^\pm (x_h,\zeta)= \frac{1}{\eps\lambda^\pm(x_h)}(\sigma \pm i \sigma^\bot)(x_h)\exp(-\lambda^\pm(x_h)\zeta),
$$
where the decay rate $\lambda^\pm$ is defined by $$(\lambda^\pm)^2=\mp i b\quad \text{and } \Re(\lambda^\pm)>0,$$
i.e.
\be
\lambda^\pm (x_h)=\lambda^\pm(y)= \frac{1\mp i \sgn(b)}{\sqrt{2}} |b(y)|^{1/2}.\label{def:lambda_pm}
\ee
Notice in particular that the decay rates  $\lambda^\pm$ vanish at $y=0$ and depend only on $y$.

Going back to the definition of $U_h^\pm$, we infer that
\be
\Ubl_h(x_h,\zeta)=\frac{U_h^+ + U_h^-}{2}=\frac{1}{2\eps}\sum_\pm \frac{(\sigma \pm i \sigma^\bot)(x_h)}{\lambda^\pm(x_h)}\exp(-\lambda^\pm(x_h)\zeta).\label{def:Ublh}
\ee
Hence, in order that $\Ubl$ is divergence free, we set
\begin{eqnarray}
\Ubl_3(x_h,\zeta)&=&-\eps \int_\zeta^\infty\Div_h \Ubl_h(x_h,\zeta')d\zeta' \nonumber \\
&=&-\frac{1}{2}\sum_\pm (\Div_h \sigma\mp i\rot_h \sigma)(x_h) (\lambda^\pm(x_h))^{-2} e^{-\lambda^\pm(x_h)\zeta}\nonumber \\
&+&\frac{1}{2}\sum_\pm (\sigma \pm i\sigma^\bot)(x_h)\cdot\frac{\nabla_h\lambda^\pm(x_h)}{(\lambda^\pm(x_h))^3}(2+ \zeta \lambda^\pm(x_h))e^{-\lambda^\pm(x_h)\zeta}.\label{def:Ubl3}
\end{eqnarray}
We have used the convention
$$
\rot_h u_h=-\Div_h u_h^\bot
$$
for two dimensional-vector fields.

The remaining flux term is then given by
\begin{eqnarray}
\ubl_{3|z=1}=\Ubl_{3|\zeta=0}(x_h)\label{Ekman_pumping}
&=&-\frac{1}{2}\sum_\pm (\Div_h \sigma\mp i\rot_h \sigma)(x_h) (\lambda^\pm(x_h))^{-2} \\
&&+\sum_\pm (\sigma \pm i\sigma^\bot)(x_h)\cdot\frac{\nabla_h\lambda^\pm(x_h)}{(\lambda^\pm(x_h))^3}.
\end{eqnarray}

\vskip1mm

We now wish to point out a particular difficulty steming from the above construction. If the Coriolis factor $b$ has vanishing points, which occurs in particular in the case of the $\beta$-plane approximation ($b(x_h)=\beta y$), then the functions $\Ubl_h,\Ubl_3$ may not be square integrable if the function $\sigma$ is arbitrary. Hence, the function $\sigma$ should vanish at a sufficiently high order near $y=0$ so that the singularity disappears. We will check that \eqref{hyp:sigma_quadra} entails that the functions $\Ubl_h,\Ubl_3$ defined by \eqref{def:Ublh}, \eqref{def:Ubl3} are square integrable. For further purposes, we also require that the function $\nabla_h \Ubl$ belongs to $L^2(\omega_h\times [0,\infty)_\zeta).$ Unfortunately, assumption \eqref{hyp:sigma_quadra} is not sufficient to ensure such a result. Thus we introduce an approximate boundary layer term, in which the low values of $b$ have been truncated.

\subsection{Estimates on the boundary layer terms}$ $

We begin with a short justification of the need for a truncation. Using the definition of $\lambda^\pm$ together with assumption \eqref{hyp:Coriolis}, we infer that if $y$ is close to zero, then
\begin{eqnarray*}
\|\nabla_{x_h} \Ubl_3(x_h)\|_{L^2([0,\infty)_\zeta)}&\leq &C\left(\frac{|D^2\sigma(x_h)|}{y^{5/4}} + \frac{|\nabla\sigma(x_h)|}{y^{9/4}} + \frac{|\sigma(x_h)|}{y^{13/4}}\right)	\\
&\leq & C y^{-5/4}.
\end{eqnarray*}
Hence $\nabla_{x_h} \Ubl_3$ does not belong to $L^2(\omega_h\times[0,\infty)_\zeta )$ in general. We thus define, for any $\delta>0,$ the function 
\be\label{truncation}
b_\delta(y)=b(y) \psi\left(\frac{|y|}{\delta}  \right),\quad y\neq 0
\ee
where $\psi\in\mathcal C^\infty((0,\infty))$ is such that
$$\ba
\psi(y)\geq \frac{1}{2}\quad\text{for } y\in(0,\infty),\\
\psi(y)=1\quad\text{if }y\geq 2,\\
\psi(y)=y^{-\alpha}\quad\text{if }y\in (0,1),
\ea
$$
for some exponent  $\alpha\in(0,1)$ to be chosen later on. Notice that with this choice of $\psi$, the function $b_\delta$ behaves like $\delta^\alpha y^{1-\alpha}$ for $y>0$ near zero. Consequently, $b_\delta$ vanishes with a weaker rate than $b$, and thus $\sigma/b_\delta$ vanishes more strongly than $\sigma/b$.

We now define approximated decay rates $\lambda_\delta^\pm$ by replacing $b$ by $b_\delta$ in the expression \eqref{def:lambda_pm}; eventually, we define approximated boundary layer terms by the formulas \eqref{def:Ublh}-\eqref{def:Ubl3}, in which the decay rates $\lambda^\pm$ have been replaced by $\lambda^\pm_\delta$.

We then have the following result:

\begin{Lem}
Assume that hypotheses \eqref{hyp:sigma_quadra}, \eqref{hyp:Coriolis} are satisfied.
Then there exists a constant $C$, depending only on $\sigma$ and $b$, such that  for all $\alpha>0, \delta>0$,
$$\begin{aligned}
   \| \Ubl_{\delta,h}\|_{L^2(\omega_h\times [0,\infty)_\zeta)}\leq \frac{C}\eps\\
 \| \Ubl_{\delta,3}\|_{L^2(\omega_h\times [0,\infty)_\zeta)}\leq C.
  \end{aligned}
$$

Additionally, if $\alpha>3/5$,  there exists a constant $C_\alpha,$ depending only on $\alpha,\sigma $ and $b$, such that for all $\delta>0$,
$$\begin{aligned}
   \|\nabla_h \Ubl_{\delta,h}\|_{L^2(\omega_h\times [0,\infty)_\zeta)}\leq \frac{C_\alpha}\eps\\
 \| \nabla_h \Ubl_{\delta,3}\|_{L^2(\omega_h\times [0,\infty)_\zeta)}\leq \frac{C_\alpha}{\delta^{3/4}}.
  \end{aligned}
$$

Moreover, for all $\delta>0$,
$$\begin{aligned}
   \|(b-b_\delta )\Ubl_{\delta,h}\|_{L^2(\omega_h\times [0,\infty)_\zeta)}\leq C \frac{\delta^{11/4}}{\eps},\\
\|  \nabla_h \Pbl_\delta\|_{L^2(\omega_h\times [0,\infty)_\zeta)}\leq C_\alpha \frac{\eps^2}{\delta^{1/4}}.
  \end{aligned}
$$
\label{lem:approx_sol}
\end{Lem}

\begin{Rmk}
The above estimates are given for the rescaled boundary layer profiles $\Ubl, \Pbl$, which are defined on $\omega_h\times [0,\infty)_\zeta$. Remember that the boundary layer part of the stationary solution is defined on $\omega_h\times [0,1]$ by
$$
\ubl(x_h,z)=\Ubl_h\left(x_h, \frac{1-z}{\eps}  \right).
$$
Hence
$$
\|\ubl_h \|_{L^2(\omega_h\times (0,1))} \leq  \eps^{1/2} \| \Ubl_h \|_{L^2(\omega_h\times [0,\infty))}.
$$
The same estimates hold for $\pbl,\ubl_3$.

\end{Rmk}

\begin{proof}

$\bullet$ \textbf{$L^2$ estimates:}
	According to \eqref{hyp:Coriolis}, and to the definition of $b_\delta$, we have
$$
\left| \lambda^\pm(x_h) \right|\neq 0\ \text{if }y\neq 0,\\
\left| \lambda^\pm(x_h) \right|\sim \sqrt{\beta |y|}\text{ as }y\to 0,
$$
and thus there exists a constant $C$ such that
$$
\left| \Re( \lambda^\pm(x_h) ) \right|^{-1},\ \left| \lambda^\pm(x_h) \right|^{-1}\leq C |y|^{-1/2}\quad\forall x, \ \forall y\in[-1,1].
$$
Similarly, for all $\delta>0$, we have, for $|y|\leq \delta$
$$
\left| \Re( \lambda^\pm_\delta(x_h) ) \right|^{-1},\ \left| \lambda^\pm_\delta(x_h) \right|^{-1}\leq C |y|^{-(1-\alpha)/2}\delta^{-\alpha/2}.
$$
If $|y|\geq \delta, $ then $\lambda^\pm_\delta(x_h)$ satisfies the same estimates as $\lambda^\pm(x_h).$
A careful computation leads to
\be\label{Ublh_L2z}
\int_0^\infty\left|\Ubl_{\delta,h}(x_h,\zeta)\right|^2\:d\zeta=\frac{1}{2\eps^2}|\sigma(x_h)|^2 \sum_\pm \frac{1}{|\lambda^\pm_\delta(y)|^2 \Re(\lambda^\pm_\delta(y))}.
\ee
 Hence we obtain
$$
 \left(\int_0^\infty\left|\Ubl_{\delta,h}(x_h,\zeta)\right|^2\:d\zeta \right)^{1/2}\leq \frac{C}{\eps}\left\{ 
\begin{array}{ll}
 |y|^\frac{5+3\alpha}{4}\delta^{-\frac{3\alpha}{4}}&\text{ if }|y|\leq \delta,\\
|y|^{5/4}&\text{ if } \delta\leq |y|\leq 1,\\
|\sigma|^2&\text{ else}.
\end{array}
\right. 
$$
Eventually,  we infer that
$$
\|\Ubl_{\delta,h} \|_{L^2(\omega_h\times [0,\infty)_\zeta)}\leq \frac{C_0}{\eps},
$$
where the constant $C_0$ depends only on $b$ and $\sigma$. Notice that the truncation does not play any role at this stage: the same arguments show that $\Ubl_h\in L^2(\omega_h\times [0,\infty)_\zeta).$

Similarly, we have
\be
 \label{Ubl3_L2z}\int_0^\infty\left|\Ubl_{\delta,3}(x_h,\zeta)\right|^2\:d\zeta \leq  C \sum_\pm \frac{|\nabla \sigma(x_h)|^2}{|\lambda^\pm_\delta(y)|^5} + \frac{|\sigma(x_h)|^2 |\nabla \lambda^\pm_\delta(y)|^2}{|\lambda^\pm_\delta(y)|^7} .
\ee
Using the definition of the decay rates $\lambda^\pm_\delta$ together with the definition of the function $\psi$, we obtain 
$$
\left|\nabla \lambda^\pm_\delta \right|= \frac{ |b_\delta'|}{2|b_\delta|^{1/2}}\leq C 
\left\{ 
\begin{array}{ll}
 |y|^{-\frac{\alpha+1 }{2}}\delta^{\alpha/2}&\text{ if }|y|\leq \delta,\\
|y|^{-1/2}&\text{ if } \delta\leq |y|\leq 1,\\
1&\text{ else}.
\end{array}
\right. 
$$
Thus
$$
\|\Ubl_{\delta,3} \|_{L^2(\omega_h\times [0,\infty)_\zeta)}\leq C_0.
$$

$\bullet$ \textbf{$H^1_h$ estimates:}

We begin with the bound on $\nabla_h \Ubl_{\delta,h}$; the calculations are very similar to the ones which led to the $L^2$ bound on $U_{\delta,3}, $ and are therefore left to the reader. In fact, the situation is even a little less singular than in the case of $U_{\delta,3}$ (we ``gain'' one integration with respect to the variable $\zeta$, and thus one factor $(\lambda^\pm_\delta)^{-1}$). The bounds on $\lambda_\delta^\pm$ and $\sigma$ entail that 
$$
\| \nabla_h \Ubl_{\delta,h}\|_{L^2(\omega_h\times [0,\infty)_\zeta)}\leq \frac{C_0}{\eps}.
$$

We now tackle the bound on $\nabla_h \Ubl_{\delta,3}$: first, differentiating equation \eqref{def:Ubl3} with respect to $x_h$, we obtain
\begin{eqnarray*}
\|\nabla_h \Ubl_{\delta,3}(x_h)\|_{L^2([0,\infty)_\zeta)}&\leq& C \left(\frac{|D^2\sigma(x_h)|}{|\lambda_\delta|^{5/2}(y)} + |\nabla\sigma(x_h)| \frac{|\d_y \lambda_\delta^\pm(y)}{|\lambda_\delta(y)|^{7/2}} \right)\\
&&+C\left( |\sigma(x_h)|	\frac{|\d_{yy}\lambda_\delta(y)|}{|\lambda_\delta(y)|^{7/2}} + |\sigma(x_h)|\frac{|\d_y\lambda_\delta(y)|^2}{|\lambda_\delta(y)|^{9/2}}\right).
\end{eqnarray*}
In the expression above, we have denoted by $|\lambda_\delta|$ the common size of $|\lambda^+_\delta|$ and $|\lambda^-_\delta|$. Notice that due to the sign change in $b$ at $y=0$, there is in general a Dirac mass at $y=0$ in the term $\d_{yy}\lambda_\delta$; more precisely, the part of $\nabla_h \Ubl_{\delta,3}$ which is not absolutely continuous with respect to the Lebesgue measure is of the type
$$
\delta_{y=0}|\sigma |\frac{|b_\delta'|}{|b_\delta|^{1/2}|\lambda_\delta|^{7/2}}=\delta_{y=0} |\sigma | |b_\delta|^{-9/4}.
$$
At this stage, the need for a truncation is clear: if $b_\delta$ is replaced by $b$, then $|\sigma| |b|^{-9/4}\sim |y|^{-1/4}  $ near $y=0$, and thus the singular part of $\nabla_h \Ubl_3$ is not well-defined in the sense of distributions. Conversely, if $\alpha>1/9$, then
$$|\sigma | |b_\delta|^{-9/4}\sim |y|^{\frac{9\alpha-1}{4}}\delta^{-\frac{9\alpha}{4}}\quad\text{as } y\to 0,$$
and thus the singular part of $\nabla_h \Ubl_{\delta,3}$ is zero.

Gathering all the terms, we deduce that
$$
\|\nabla_h \Ubl_{\delta,3}(x_h)\|_{L^2_\zeta}\leq C
\left\{ 
\begin{array}{ll}
 |y|^{-\frac{5(1-\alpha)}{4}}\delta^{-\frac{5\alpha}{4}}&\text{ if }|y|\leq \delta,\\
|y|^{-5/4}&\text{ if } \delta\leq |y|\leq 1,\\
|\sigma(x_h)| + |\nabla \sigma(x_h) | + |D^2\sigma(x_h)|&\text{ else}.
\end{array}
\right. 
$$
Thus $\nabla_h \Ubl_{\delta,3}\in L^2(\omega_h\times [0,\infty))$ if and only if $\alpha>3/5,$ and in this case there exists a constant $C_\alpha$, depending on $\sigma, b$ and $\alpha$, such that for all $\delta>0$
$$
\|\nabla_h \Ubl_{\delta,3}(x_h)\|_{L^2(\omega_h\times [0,\infty))}\leq \frac{C_\alpha}{\delta^{3/4}}.
$$

$\bullet$ \textbf{Error estimates:}
First, by definition of $b_\delta$, we have
$$
\left\| (b-b_\delta) \Ubl_{\delta, h} \right\|_{L^2}^2=\int_{\omega_h\cap\{|y|\leq 2\delta\}}\int_0^\infty |b(y)-b_\delta(y)|^2 \left| \Ubl_{\delta, h}(x,y,\zeta) \right|^2 d\zeta\:dy\:dx.
$$
Notice that for all $y\in\bR\setminus\{0\}$,
\begin{eqnarray*}
|b(y)-b_\delta(y)|&=&|b(y)|\left|1 - \psi\left(\frac{y}{\delta}\right)\right|\\
&=& \mathbf 1_{|y|\leq \delta}|b(y)|\left(\frac{\delta^\alpha}{|y|^\alpha} - 1\right) +  \mathbf 1_{\delta \leq |y|\leq 2\delta}|b(y)|\left|1 - \psi\left(\frac{y}{\delta}\right)\right|\\
&\leq & C \left(\mathbf 1_{|y|\leq \delta}|y|^{1-\alpha}(\delta^\alpha-|y|^\alpha) +  \mathbf 1_{\delta \leq |y|\leq 2\delta}|y|\right)\\
&\leq & C \mathbf 1_{|y|\leq 2\delta}|y|^{1-\alpha}\delta^\alpha.	
\end{eqnarray*}
Using \eqref{Ublh_L2z}, we infer
\begin{eqnarray*}
&&\left\| (b-b_\delta) \Ubl_{\delta, h} \right\|_{L^2(\omega_h\times [0,\infty)_\zeta)}^2\\&\leq & \frac{C}{\eps^2}\int_{x\in\bT}\int_{|y|\leq 2 \delta} 
|y|^{2(1-\alpha)} \delta^{2\alpha} |y|^4 |y|^{-3/2} \left|\frac{y}{\delta}\right|^{3\alpha/2}\:dy\\
&\leq & \frac{C}{\eps^2} \int_{|y|\leq 2 \delta} |y|^{(9-\alpha)/2} \delta^{\alpha/2}\:dy\\
&\leq & \frac{C\delta^{11/2}}{\eps^2} .
\end{eqnarray*}
There remains to evaluate $ \Pbl_\delta$. By definition,
$$
\Pbl_\delta= - \eps^3 \Div_h \Ubl_{h,\delta}.
$$
Using the same kinds of calculations as the ones which led to the bound on $\nabla_h \Ubl_{\delta,3}$,  we deduce that 
$$
\| \nabla_h \Pbl_\delta\|_{L^2(\omega_h\times [0,\infty)_\zeta)}\leq C \frac{\eps^2}{\delta^{1/4}}.
$$

\end{proof}

\subsection{Error estimates in the case $\nu_h\neq 0$ and conditions on the parameter $\delta$}$ $

If $\nu_h\neq 0$, we keep the construction of the previous paragraph, and we merely treat the viscous terms as error terms. The function $\ubl_{\delta,h}$ is an approximate solution of the horizontal part of equation \eqref{rescaled}, with the error term
$$
\frac{1}{\eps}(b-b_\delta) (\ubl_{\delta,h})^\bot - \nu_h \Delta_h \ubl_{\delta,h} + \nabla_h \pbl_\delta.
$$
According to the estimates of the previous paragraph (see Lemma \ref{lem:approx_sol}), we have
$$
\left\| \frac{1}{\eps}(b-b_\delta) (\ubl_{\delta,h})^\bot \right\|_{L^2(\omega)}\leq C \frac{\delta^{11/4}}{\eps^{3/2}}
$$
and
$$
\ba
\|\nu_h \Delta_h \ubl_{\delta,h}  \|_{L^2([0,1], H^{-1}(\omega_h))}\leq C\frac{\nu_h}{\sqrt{\eps}},\\
\|  \nabla_h \pbl\|_{L^2(\omega)}\leq C\frac{\eps^{5/2}}{\delta^{1/4}} .
\ea
$$
Recall that because of the boundary layer scaling, there is a factor $\eps^{1/2}$ between the $L^2$ norms of $\ubl_\delta$ and $\Ubl_\delta$.

Hence, in order that the conditions of Theorem \ref{thm:stat} are satisfied, the numbers $\eps,\nu_h,\delta$ should verify
\be\label{cond:delta1}
\eps^{10}\ll\delta\ll \eps^{6/11},\quad \nu_h \ll  \eps.
\ee

On the other hand, $\ubl_{\delta,3}$ is an approximate solution of the vertical component of equation \eqref{rescaled}, with an error term equal to 
$$-\nu_h \Delta_h \ubl_3.$$
The estimates of the previous paragraph entail that 
$$
\left\| -\nu_h \Delta_h \ubl_3\right\|_{L^2([0,1], H^{-1}(\omega_h))}\leq C \frac{\nu_h\sqrt{\eps}}{\delta^{3/4}}.
$$
In order that the condition of Theorem \ref{thm:stat} is satisfied, the parameter $\delta$ must be chosen so that 
\be\label{cond:delta2}
\delta\gg \nu_h^{2/3} \eps^2.
\ee
Notice that if $\nu_h=O(1)$ and $\eps=o(1)$, we always have 
$$
\nu_h^{2/3} \eps^2\ll \eps^{6/11}.
$$
Hence it is always possible to choose a parameter $\delta$ which matches the above conditions.

Further conditions on the parameter $\delta$ will be given in the next section.
When the conditions \eqref{cond:delta1}, \eqref{cond:delta2} are satisfied, the couple $(\ubl_\delta,\pbl_\delta)$
is an approximate solution of equation \eqref{rescaled} in the sense of Theorem \ref{thm:stat}. Furthermore, $\ubl_\delta$ satisfies the horizontal part of the boundary condition \eqref{top} at $z=1$; $\ubl_3$, on the other hand, does not satisfy the non penetration condition at $z=1$. Hence, we construct in the next section an interior term, which is also an approximate solution of \eqref{rescaled}, and which lifts the trace of $\ubl_3$ at $z=1$.

Notice that $\ubl$ also has a non-vanishing trace at $z=0$; however, this trace is exponentially small on the set where $b$ is bounded away from zero, and can thus be lifted thanks to an exponentially small corrector. This will be taken care of after the construction of the interior term $\uint$, in the last paragraph of the next section.

\section{The interior part of the stationary solution}

\label{sec:int}

In this section, we construct a stationary solution $\uint$ of equation \eqref{rescaled}, which is such that $\uint + \ubl$ satisfies the boundary conditions \eqref{bottom}, \eqref{top}. Going back to equation \eqref{rescaled}, it can be readily checked that the function $\uint$ should satisfy the system
\be
\begin{aligned}
b(y) (\uint_h)^\bot + \nabla_h p=0 ,\\
\d_z p=0,\\
\Div \uint=0,
\end{aligned}
\label{eq:uint}\ee
together with the boundary conditions
\be\label{CL:uint}
\begin{aligned}
\d_z \uint_{h|z=1}=0, \quad\uint_{3|z=1}=-\ubl_{3|z=1},\\ 
\d_z \uint_{h|z=0}=0, \quad\uint_{3|z=0}=0.
\end{aligned}
\ee
We recall that since the function $\ubl_3$ depends on the small parameter $\delta$, the function $\uint$ also depends on $\delta$ in general, and thus will be denoted by $\uint_\delta$ in the sequel. Hence we also investigate the asymptotic behaviour of $\uint_\delta$ as $\delta\to 0.$

It turns out that the solution of the sytem \eqref{eq:uint}-\eqref{CL:uint} is unique, up to a function of the type $(v(y),0,0).$ Hence we give in this paragraph a straightforward way of building the solution, and then we derive $L^2$ estimates on the function $\uint_\delta$. The main result of this section is the following:
\begin{Lemma}
Assume that assumptions \eqref{hyp:sigma_quadra}-\eqref{hyp:Coriolis} are fulfilled. Then there exists a solution $\uint_\delta\in L^2(\omega)$ of the system \eqref{eq:uint}. Moreover, there exists a positive constant $C$, depending only on $\sigma $ and $b$, such that
$$
\|\uint_\delta\|_{L^2(\omega)}\leq C\quad\forall \delta>0.
$$
\end{Lemma}

\subsection{Construction of $\uint_\delta$} $ $

To begin with, we differentiate the first equation of \eqref{eq:uint} with respect to $z$, and we obtain
$$
b(y) \d_z (\uint_{\delta,h})^\bot=0.
$$
Since $\uint_\delta$ is divergence-free, we infer that $\d_{zz}\uint_{\delta,3}=0$. Hence the third component $\uint_{\delta,3}$ is uniquely determined; in order to lighten the notation, set $$
w_\delta(x_h)=-\ubl_{\delta,3|z=1}(x_h).
$$
 We have
$$
\uint_{\delta,3}(x_h,z)=z w_\delta(x_h).
$$

Then, taking the two-dimensional curl  of the first equation in \eqref{eq:uint}, we derive
$$
\rot_h (b(\uint_{\delta,h})^\bot)=\Div_h (b\uint_{\delta,h})=0.
$$
Since the Coriolis factor only depends on the latitude $y$, we are led to
$$
b'(y) \uint_{\delta,2}=-b(y) \Div_h \uint_{\delta,h}=+ b(y)\d_z \uint_{\delta,3}=b(y) w_\delta(x_h). 
$$
Consequently, the second component is also uniquely determined. In the case when $b(y)=\beta y$, one has in particular
$$
\uint_{\delta,2}(x_h)=y w_\delta(x_h).
$$
This equation is known as the \textbf{Sverdrup relation} (see \cite{Pedlosky1,Pedlosky2}).

There remains to compute the first component of $\uint$; the divergence-free condition entails that
\begin{eqnarray*}
 \d_x \uint_{\delta,1}&=&-\d_y \uint_{\delta,2} -\d_z \uint_{\delta,3}=-\d_y \left(\frac{b}{b'}w_\delta  \right) - w_\delta \\
&=&-\left[ \d_y \left(\frac{b}{b'} \right) + 1\right] w_\delta -\frac{b}{b'}\d_y w_\delta\\
&=&-\left(2- \frac{bb''}{{b'}^2} \right)w_\delta - \frac{b}{b'}\d_y w_\delta.
\end{eqnarray*}
Notice that this equation has a solution in $\omega_h$ if and only if the right-hand side has zero average in $x$, for all $y$. This is satisfied in particular if 
\be
\int_{\bT} w_\delta(x,y)\:dx=0\quad\forall y.\label{compatibility2}
\ee
We assume that this assumption is satisfied for the time being, and we will prove that it is in fact equivalent to \eqref{compatibility}.
Integrating the equality giving $\d_x \uint_{\delta,1}$ with respect to $x$, we deduce that $\uint_{\delta,1}$ is defined up to a function of $y$ only, provided \eqref{compatibility2} is satisfied.

Now, let us compute $w_\delta$ in terms of $\sigma$ and $b$. Using equation \eqref{Ekman_pumping}, we infer that
$$
w_\delta(x_h)=\frac{1}{2}\sum_\pm \left( \Div_h \sigma \mp i \rot_h \sigma \right)\frac{1}{(\lambda^\pm_\delta)^2} - \sum_\pm (\sigma \pm i \sigma^\bot) \cdot \frac{\nabla_h \lambda^\pm_\delta}{(\lambda^\pm_\delta)^3}.
$$
By definition of $\lambda^\pm$ (see \eqref{def:lambda_pm}), we have
$$
\nabla \lambda^\pm_\delta= \left( 0, \frac{1\mp i \sgn(b)}{2\sqrt{ 2}} \frac{\sgn(b) b'_\delta}{|b_\delta|^{1/2}}\right).
$$
Hence
\be\label{def:w}
w_\delta(x_h)=\frac{1}{b_\delta}\rot_h \sigma +\frac{1}{b_\delta^2}\sigma^\bot \cdot \nabla b_\delta= \frac{\d_x \sigma_2}{b_\delta}- \d_y\left(\frac{\sigma_1}{b_\delta}\right).
\ee
(Recall that $b_\delta$ only depends on the latitude $y$.)

We now prove the equivalence of \eqref{compatibility} and \eqref{compatibility2}. It is clear that $\eqref{compatibility}\Rightarrow\eqref{compatibility2}. $ Conversely,  if \eqref{compatibility2} is satisfied, then 
\eqref{def:w} leads to the existence of a constant $\alpha_\delta\in\bR$ such that
$$
\int_{\bT} \frac{\sigma_1(x,y)}{b_\delta(y)}\:dx = \alpha_\delta \quad\forall y.
$$
Since $\sigma_1$ vanishes quadratically near $y=0$, we deduce that the left-hand side of the above equality vanishes at least linearly near $y=0$. Consequently, $\alpha_\delta=0$ for all $\delta,$ and thus \eqref{compatibility} is satisfied.

\subsection{Bounds on $\uint$}$ $

We begin with a bound on the function $w_\delta$ given by \eqref{def:w}.
We recall that
$$
b(y)\sim \beta y\quad\text{near }y=0,
$$
and 
$$\begin{aligned}
   \sigma_1(x,y)=\frac{1}{2}\d_y^2\sigma_1(x,0) y^2 + O(y^3)\quad\text{as } y\to 0,\\
\d_x \sigma_2 (x,y)=O(|y|^2)\quad\text{as } y\to 0.
  \end{aligned}
$$
Thus
$$
 \d_y\left(\frac{\sigma_1}{b_\delta}\right)=\frac{b_\delta\d_y \sigma_1-\sigma_1\d_y b_\delta}{b_\delta^2}= O(|y|^\alpha\delta^{-\alpha})\quad \text{for }y\to0, \ |y|\leq \delta.
$$
The exponent $\alpha$ was introduced in the previous section, see \eqref{truncation}.

Consequently, there exists a constant $C$ (independent of $\delta$) such that
$$
\| w_\delta\|_{L^2(\omega_h)}\leq C.
$$
This entails immediately that $\uint_{\delta,3}$ and $\uint_{\delta,2}$ are bounded in $L^2(\omega)$, uniformly in $\delta.$

As for $\uint_{\delta,1}$, we have, by definition
\begin{eqnarray*}
\d_x \uint_{\delta,1}&=& -\d_y\left(\frac{b \d_x \sigma_2}{b'b_\delta} \right) - \frac{\d_x\sigma_2}{b_\delta}+ \d_y\left(\frac{b}{b'} \d_y \frac{\sigma_1}{b_\delta}\right) +  \d_y \frac{\sigma_1}{b_\delta}\\
&=&-\d_y \left(\frac{\d_x\sigma_2}{\psi(\frac{\cdot}{\delta}) b'}\right)- \frac{\d_x\sigma_2}{b_\delta}+\d_y \frac{\d_y\sigma_1}{b'\psi(\frac{\cdot}{\delta})}- \d_y\left(\frac{\sigma_1}{b_\delta}\left(\frac{b b_\delta'}{b'b_\delta}-1\right)\right).
\end{eqnarray*}
Integrating with respect to the variable $x$, we deduce that
$$
\uint_{\delta,1}=- \d_y \left(\frac{\sigma_2}{\psi b'}\right)- \frac{\sigma_2}{b_\delta}+\d_y \frac{\d_yS_1}{b'\psi}- \d_y\left(\frac{S_1}{b_\delta}\left(\frac{b b_\delta'}{b'b_\delta}-1\right)\right),
$$
where $S_1(x,y)=\int_0^x \sigma_1(x',y)\:dx'$.
Using the definition of $b_\delta$, we obtain
$$
\frac{b b_\delta'}{b'b_\delta}-1=\frac{1}{\delta}\frac{b\psi'(\frac{\cdot}{\delta})}{b'\psi(\frac{\cdot}{\delta})}.
$$
It can be checked that the function in the right-hand side is $\mathcal C^\infty$ on $(0,\infty)$ and bounded, together with all its derivatives. Moreover, its support is included in $[0,2\delta].$ As a consequence, the term 
$$ \d_y\left(\frac{S_1}{b_\delta}\left(\frac{b b_\delta'}{b'b_\delta}-1\right)\right)$$
is $o(1)$ in $H^1(\omega)$ as $\delta\to 0$. The other terms can be evaluated in a similar fashion. Using the assumptions on $\sigma$ and $b$ together with the definition of $\psi$, we deduce that there exists a constant $C[\sigma]$ such that
\be
\|\uint_\delta\|_{L^2(\omega)}\leq C[\sigma].
\ee

\vskip2mm

$\bullet$ We now derive estimates in $L^2([0,1], H^1(\omega_h))$, which are needed to bound the error term $\nu_h\Delta_h \uint_\delta$. First, using the definition of $b_\delta$ together with assumptions \eqref{hyp:sigma_quadra}, \eqref{hyp:Coriolis}, it can be proved that
$$
\d_y w_\delta=O(y^{\alpha-1} \delta^{-\alpha})\quad\text{as }y\to 0, |y|\leq \delta.
$$
Hence $\d_yw_\delta\in L^2(\omega)$ (recall that $\alpha>3/5>1/2$) and 
$$
\| \d_yw_\delta\|_{L^2(\omega)}=O(\delta^{-1/2}).
$$
The term $\d_x w_\delta$, on the other hand, is bounded in $L^2(\omega)$, uniformly in $\delta.$
Consequently, there exists a constant $C$, depending only on $\sigma, b$ and $\alpha$, such that
$$
\| \nabla_h \uint_{\delta,3}\|_{L^2(\omega)}\leq \frac{C}{\delta^{1/2}}.
$$

Similarly, we prove that $\d_y u_{\delta,2}=O(|y|^\alpha \delta^{-\alpha})$ for $y$ in a neighbourhood of zero, and thus
there exists a constant $C$ such that
$$
\| \nabla_h \uint_{\delta,2}\|_{L^2(\omega)}\leq C.
$$

We now tackle the term $u_{\delta,1}$; using either the expression of $\d_x u_{\delta,1}$ in terms of $w_\delta$ or the final definition in terms of $\sigma_2$ and $S_1$, it can be checked that 
$$
\d_y u_{\delta,1}=O(y^{\alpha-1} \delta^{-\alpha})\quad\text{as }y\to 0, |y|\leq \delta.
$$
The largest terms are those coming from $S_1$ (or from $b\d_yw_\delta/b'$); for instance, the above calculations show that
$$
\frac{b}{b'}\d_y w_\delta=O(|y|^\alpha \delta^{-\alpha});
$$
since one power of $y$ is lost with each differentiation with respect to $y$, we obtain the desired bound on $u_{\delta,1}$. Eventually, we are led to 
$$
\| \nabla_h \uint_{\delta,1}\|_{L^2(\omega)}\leq \frac{C}{\delta^{1/2}}.
$$

\vskip1mm

$\bullet$ Notice that
$$
\uint_\delta\to \uint \quad \text{in }L^2(\omega)
$$
as $\delta\to 0$, where $\uint$ is the function defined by the same expressions as $\uint_\delta$, but replacing every occurrence of $w_\delta$ by 
$$
w=\frac{\rot_h\sigma}{b}+ \frac{\sigma^\bot \cdot \nabla b}{b^2}.
$$
By definition of $b_\delta$, $w$ and $w_\delta$ coincide on the set $\{|y|\geq 2\delta\}$. Moreover, $w$ is bounded in $L^2$ and $w, y \d_yw$ have finite limits  as $y\to0,$ while
$$
\int_{\bT}\int_{|y|\leq \delta} |w_\delta|^2 + |y|^2 |\d_y w_\delta|^2=o(1).
$$
 
Consequently, $w_\delta$ (resp. $b\d_yw_\delta$) converges towards $w$ (resp $b\d_y w$) in $L^2(\omega)$ as $\delta\to 0$. The convergence of $\uint_\delta$ follows. However, in general, $\uint$ does not belong to $H^1(\omega)$, except if the surface stress $\sigma$ vanishes at sufficiently high order.

\subsection{Proof of Theorem \ref{thm:stat}}$ $

Let us first evaluate the error terms in equation \eqref{NS}.
To begin with, notice that $\d_{zz}\uint_\delta=0$, so that there is no error term associated with the vertical Laplacian. Consequently, the only error terms in equation \eqref{rescaled} are those coming from the term $\nu_h \Delta_h \uint_\delta.$

According to the $H^1$ estimates of the previous paragraph, we have
$$
\ba
\left\| \nu_h \Delta_h \uint_{\delta,h}\right\|_{L^2([0,1], H^{-1}(\omega_h))}\leq C \frac{\nu_h}{\sqrt{\delta}},\\
\left\| \nu_h \Delta_h \uint_{\delta,3}\right\|_{L^2([0,1], H^{-1}(\omega_h))}\leq C \frac{\nu_h}{\sqrt{\delta}}.
\ea
$$
In order that the conditions of Theorem \ref{thm:stat} are satisfied, we have to choose the parameter $\delta$ so that $\delta\gg \nu_h.$ We recall that $\delta,\nu_h$ should also satisfy \eqref{cond:delta1}, \eqref{cond:delta2}. Thus the new conditions on $\delta $, $\nu_h$ are
\be\label{cond:delta3}
\max(\nu_h, \eps^{10}, \nu_h^{2/3}\eps^2) \ll \delta\ll \eps^{6/11},\quad  \nu_h\ll \eps.
\ee

\vskip1mm

$\bullet$ The proof of Theorem \ref{thm:stat} is now almost complete. There only remains to take care of the  boundary conditions: indeed, as we have explained at the end of the previous section, the trace of $\d_z \ubl_{\delta,h}$ and $\ubl_{\delta,3}$ is non zero at $z=0$. Hence, we define a corrector $\vint_\delta$, which is small in $H^1$, and which lifts the remaining boundary conditions. The result is the following:

\begin{Lem}
Assume that $\nu_h=o(\eps)$ and that there exists $\kappa\in(1,2)$ such that $\delta\gtrsim \eps^\kappa$. 

Then there exists a divergence free function $\vint_\delta$, such that $\vint_\delta=o(1)$ in $L^2(\omega)$, which satisfies the conditions
$$\ba
\d_z\vint_{\delta,h|z=1}=0, \quad\vint_{\delta,3|z=1}=0,\\
\d_z\vint_{\delta,h|z=0}=-\d_z\ubl_{h|z=0}, \quad\vint_{\delta,3|z=0}=-\ubl_{3|z=0}.
\ea
$$	
Furthermore, we can choose the parameter $\alpha$ of the truncation function $\psi$ so that
$$\ba 
\frac{1}{\eps}b e_3\wedge \vint_\delta , - \eps \d_{zz}\vint_h=o(1)\quad   \text{in }L^2(\omega),\\
-\eps\d_{zz}\vint_3=o(\eps^{-1})\quad \text{in }L^2(\omega),\\
\sqrt{\nu_h} \Delta_h \vint_h =o(1)\quad   \text{in }L^2([0,1]), H^{-1}(\omega_h)),\\
\sqrt{\nu_h}  \Delta_h \vint_3 =o(\eps^{-1})\quad   \text{in }L^2([0,1]), H^{-1}(\omega_h)).
\ea
$$

\label{lem:stop}
\end{Lem}

Before proving the lemma, let us complete the proof of Theorem \ref{thm:stat}: we choose a parameter $\delta$ which matches the conditions of Lemma \ref{lem:stop} together with \eqref{cond:delta3}. Notice that the choice $\delta=\eps $ works.
We set 
$$
\ustat=\ubl_\delta + \uint_\delta+ \vint_\delta;
$$
by construction, $\ustat$ satisfies the boundary conditions \eqref{bottom}, \eqref{top}, and it is an approximate solution of equation \eqref{rescaled}. The bounds on $\uint_\delta$ and $\ubl_\delta$ were proved in the previous paragraphs. Hence Theorem \ref{thm:stat} is proved.

\begin{proof}[Proof of Lemma \ref{lem:stop}]
Throughout the proof, we drop all indices $\delta$ in order not to burden the notation.

The construction of the corrector $\vint$ follows the one given in Lemma 1 in Appendix B of \cite{DSR}: setting
$$
\phi_h:=-\d_z\ubl_{h|z=0},\quad\phi_3:=-\ubl_{3|z=0},
$$	
we define
$$
\vint_h=\frac{(1-z)^2}{2}\phi_h + \nabla_h \chi,
$$
where the potential $\chi\in H^2(\omega_h)$ is defined by
\begin{eqnarray*}
\Delta_h \chi&=&\int_0^1 \Div_h\vint_h - \frac{1}{6}\Div_h \phi_h\\	
&=&- [\vint_3]^{z=1}_{z=0}- \frac{1}{6}\Div_h \phi_h= \phi_3 - \frac{1}{6}\Div_h \phi_h.
\end{eqnarray*}
We will check later on that the function $\phi_3$ has zero mean value on $\omega_h$, so that $\chi$ is well-defined. The third component of $\vint$ is then determined by
$$
\vint_3(x_h,z)=-\int_z^1\d_z \vint_3(x_h,z')dz'= \int_z^1 \Div_h \vint_h(x_h,z')\:dz'.
$$
By construction, $\vint$ is divergence free and satisfies the correct boundary conditions. There remains to evaluate $\vint$ in $L^2(\omega)$ and $L^2([0,1],H^1(\omega_h)).$

The boundary conditions $\phi_h, \phi_3$ are given by
\begin{eqnarray*}
\phi_h&=& -\frac{1}{2\eps}\sum_\pm(\sigma \pm i \sigma^\bot)\exp\left(-\frac{\lambda^\pm}{\eps}\right),\\
\phi_3&=& -\frac{1}{2}\sum_\pm (\Div_h \sigma \mp i \rot_h \sigma)\frac{1}{(\lambda^\pm)^2}\exp\left(-\frac{\lambda^\pm}{\eps}\right)\\	
&&+ \frac{1}{2}\sum_\pm (\sigma\pm i \sigma^\bot)\cdot \frac{\nabla_h \lambda^\pm}{(\lambda^\pm)^3}\left(2 + \frac{\lambda^\pm}{\eps}\right)\exp\left(-\frac{\lambda^\pm}{\eps}\right).
\end{eqnarray*}
Recall that in the expressions above, the functions $\lambda^\pm$ are in fact $\lambda^\pm_\delta$.
Notice that
$$
\phi_3=\Div_h \varphi,
$$
where
$$
\varphi=-\frac{1}{2}\sum_\pm(\sigma \pm i \sigma^\bot)\frac{1}{(\lambda^\pm)^2}\exp\left(-\frac{\lambda^\pm}{\eps}\right);
$$
this proves that $\phi_3$ has zero mean value on $\omega_h$, and will be used several times in the proof.

We now derive three type of estimates: first, estimates of $\Div_h\phi_h$ and $\phi_3$ in $L^2(\omega_h)$ will yield $H^2(\omega_h)$-bounds on $\chi$, and thus bounds in $L^2([0,1],H^1(\omega_h))$ for the function $\vint_h$, and in $L^2(\omega)$ for the function $\vint_3$. Then, estimates of $\phi_h$ and $\varphi$ will provide $L^2(\omega)$-bounds on $\vint_h$. Eventually, $L^2$ estimates of $\nabla_h\phi_3,D^2_h\phi_h$ will allow us to derive bounds on $\vint_3$ in $L^2([0,1], H^1(\omega_h))$.

\vskip1mm

$\bullet$\textbf{ Estimates of $\Div_h \phi_h$ and $\phi_3$ in $L^2(\omega_h)$:}

The main difficulty lies in the fact that $\lambda_\delta^\pm$ does not have the same behaviour for $|y|\leq \delta$ and $|y|\geq \delta$. We merely explain how the term $\Div_h \phi_h$ is evaluated; the treatment of the term $\phi_3$ is left to the reader.

If $|y|\geq 1$, we have 
$$
(\lambda^\pm)^2= \mp i b(y),\quad\text{with } |b(y)|\geq C \text{ (see \eqref{hyp:Coriolis})}.
$$
Thus 
$$
\left|\exp\left(-\frac{\lambda^\pm}{\eps}\right)\right| \leq \exp\left(-\frac{C}{\eps}\right),
$$
and 
$$
\int_{|y|\geq 1}\int_{x\in \bT}|\Div_h \phi_h(x,y)|^2\:dx\:dy\leq \frac{C}{\eps^4}\|\sigma\|_{H^1(\omega_h)}^2\exp\left(-\frac{2C}{\eps}\right).
$$
On the set where $\delta\leq |y|\leq 1$, the assumptions on the truncation function $\psi$ entail that there exists a constant $c$ such that
$$
\ba
c^{-1}|y|^{1/2}\leq\Re(\lambda^\pm_\delta(y)), |\lambda^\pm_\delta(y)| \leq c |y|^{1/2},\\
|\d_y \lambda^\pm_\delta(y)|\leq c |y|^{-1/2}.
\ea
$$
As a consequence,
\begin{eqnarray*}
&&\int_{\delta\leq |y|\leq 1}\int_{x\in\bT}|\Div_h \phi_h(x,y)|^2\:dx\:dy	\\
&\leq & \frac{C}{\eps^2}\int_\delta^1 |y|^2 \exp\left(-2c\frac{\sqrt{y}}{\eps}\right)\:dy+ \frac{C}{\eps^4}\int_\delta^1 |y|^4 \frac{1}{|y|}\exp\left(-2c\frac{\sqrt{y}}{\eps}\right)\:dy\\
&\leq &C\eps^4.
\end{eqnarray*}
There remains to treat the set where $|y|\leq \delta$; because of the truncation function $\psi$, this part is the most complicated. The definition of the function $\psi$ and the fact that $b(y)\sim \beta y$ for $y$ close to zero entail that
$$
\ba
c^{-1}|y|^{\frac{1-\alpha}{2}}\delta^\frac{\alpha}{2}\leq |\lambda^\pm_\delta(y)|, \Re(\lambda^\pm_\delta(y))\leq c |y|^{\frac{1-\alpha}{2}}\delta^\frac{\alpha}{2},\\
|\d_y \lambda^\pm_\delta(y)|\leq c |y|^{-\frac{1+\alpha}{2}}\delta^\frac{\alpha}{2}.
\ea
$$
Thus, for instance
\begin{eqnarray*}
&&\int_{|y|\leq \delta}\int_{x\in \bT} \left|\Div_h \sigma \exp\left(-\frac{\lambda^\pm}{\eps}\right)\right|^2\\
&\leq & C \int_0^\delta |y|^2\exp\left( - c \frac{|y|^{\frac{1-\alpha}{2}}\delta^\frac{\alpha}{2}}{\eps}\right)\:dy\\
&\leq & C\left(\frac{\eps^\frac{2}{1-\alpha}}{\delta^\frac{\alpha}{1-\alpha}}\right)^3.
\end{eqnarray*}
The other terms in $\Div_h \phi_h$ are evaluated in the same way. Gathering all the terms, we infer that
$$
\| \Div_h \phi_h\|_{L^2(\omega_h)}\leq C\left(\eps^{-2}\exp(-C/\eps) + \eps^2 + \eps^{-1} \left(\frac{\eps^\frac{1}{1-\alpha}}{\delta^\frac{\alpha}{2(1-\alpha)}}\right)^3\right).
$$
The corresponding term in $\sqrt{\nu_h}\Delta_h \vint$ should be $o(1)$ in $L^2([0,1], H^{-1}(\omega_h))$; hence the parameters $\eps,\nu_h,\delta$ must satisfy
$$
\ba
\sqrt{\nu_h}\left(\eps^{-2}\exp(-C/\eps) + \eps^2 \right)=o(1),\\
\sqrt{\nu_h }\eps^{-1} \left(\frac{\eps^\frac{1}{1-\alpha}}{\delta^\frac{\alpha}{2(1-\alpha)}}\right)^3=o(1).
\ea
$$
It is obvious that for $\nu_h,\eps\ll 1$, the first condition is always satisfied. The second condition reads
$$
\delta\gg\nu_h^\frac{1-\alpha}{3\alpha}\eps^\frac{4+2\alpha}{3\alpha}.
$$
Since $\nu_h\ll\eps $ (see \eqref{cond:delta3}), we always have 
$$
\nu_h^\frac{1-\alpha}{3\alpha}\eps^\frac{4+2\alpha}{3\alpha}\ll \eps^\frac{5+\alpha}{3\alpha},
$$
and  $2<(5+\alpha)/3\alpha $ since $\alpha<1$. Consequently, provided that $\nu_h\ll\eps,$ we have 
$$
\nu_h^\frac{1-\alpha}{3\alpha}\eps^\frac{4+2\alpha}{3\alpha}\ll\eps^2.
$$
Hence, if $\delta\gtrsim\eps^2$, $\sqrt{\nu_h}\Div_h\phi_h=o(1)$ in $L^2(\omega_h).$ It can be checked that $\phi_3$ satisfies the same property. The same estimates also prove that
$$\ba
\eps \d_{zz}\vint_3=o(\eps^{-1})\quad\text{in }L^2(\omega),\\
\vint_3=o(1)\quad\text{in }L^2(\omega).
\ea$$
Similarly, we show that $\phi_h, \varphi=o(\eps)$ in $L^2(\omega_h)$ as long as $\delta \gtrsim\eps^2$, and thus  $\vint_h=o( \eps )$ in $L^2.$ Notice that this is not entirely sufficient to prove the assertion of the Lemma if the Coriolis factor $b$ is unbounded. However, using the fact that $\phi_h$ and $\varphi$ decay like $\exp(-|b|^{1/2}/\eps)$ for $|y|\geq 1$, it can be easily proved that 
$$
\Delta_h(b\chi) = b\Delta_h\chi + 2 b'\d_2\chi + b''\chi=o(\eps)\quad\text{in }H^{-1}(\omega_h).
$$
Hence $b\chi=o(\eps)$ in $H^1(\omega_h),$ and $b\nabla\chi=\nabla(b\chi)-b'\chi=o(\eps)$ in $L^2$. Eventually, we infer that $b\vint_h=o(\eps)$ in $L^2(\omega)$.

\vskip1mm

$\bullet$\textbf{ Estimates of $D^2\phi_h$ and $\nabla_h\phi_3$ in $L^2(\omega_h)$:}

Calculations similar to the ones led above show that
$$
\| D^2\phi_h\|_{L^2(\omega_h)}\leq  C\|\sigma	\|_{H^2(\omega_h)}\left(\frac{\exp(-C/\eps)}{\eps^3} + 1 +\left(\frac{\eps}{\sqrt{\delta}} \right)^\frac{\alpha}{1-\alpha}\right).
$$
And if $\delta\gtrsim\eps^2$, $\nu_h\ll \eps,$ then the right hand side is $o(\eps^{-1} \nu_h^{-1/2}).$

The term $\nabla_h\phi_3$ is the most singular of all, and eventually prevents us from taking $\delta\gtrsim \eps^2;$ indeed, it can be proved that 
$$
\| \nabla_h \phi_3\|_{L^2(\omega_h)}\leq  C\|\sigma\|_{H^2(\omega_h)}\left(\frac{\exp(-C/\eps)}{\eps^2} + 1+ \delta^{-1/2} +\frac{\eps^\frac{2(2\alpha-1)}{1-\alpha}}{\delta^\frac{\alpha}{1-\alpha}} \right).
$$
In order that the right hand side is $o((\eps\sqrt{\nu_h})^{-1}),$ we must have 
$$
\delta\gg \eps^\frac{5\alpha-1}{2\alpha}.
$$
Since $\delta\gtrsim \eps^\kappa$ for some $\kappa\in(1,2)$, we choose $\alpha\in(0,1)$ such that $\frac{5\alpha-1}{2\alpha}>\kappa$. We then infer that $\nabla_h \phi_3=o((\eps\sqrt{\nu_h})^{-1})$ in $L^2(\omega_h)$, and thus $\sqrt{\nu_h}\Delta_h \vint_3=o(\eps^{-1})$ in $L^2([0,1], H^{-1}(\omega_h)).$

\end{proof}

\section{Two-dimensional propagation}
\label{sec:2D}

We recall that throughout this section and the following, we assume that $b(x_h)=\beta y$, and that $\omega_h=\bT\times \bR.$ The object of this section is to prove  the ``two-dimensional part'' of Theorem \ref{thm:propagation}. In particular, we prove that a two-dimensional perturbation of the solution $\ustat$ creates waves, propagating at a speed of order $\eps^{-1}$, with frequencies given by 
$$\beta \frac{k}{|k|^2 + |\xi_y|^2},$$
where $(k,\xi_y)$ is the wavelength.

A consequence of our result is that if $\ustat$ is initially perturbed by a two-dimensional function $u^0$ such that $u^0=O(1)$ in $L^2$ and such that the $x$-average of $u^0 $ is zero (i.e. $u^0$ has no Fourier mode corresponding to $k=0$), then the solution of \eqref{rescaled} with initial data $\ustat+ u^0$ becomes close to $\ustat $ for finite times, with an error term which is $o(1)$ in $L^2([T_0,T]\times\omega)$ for all $T>T_0>0$.

\begin{Def}
Denote by $\Ph: L^2(\omega_h)^2\to L^2(\omega_h)^2$  the projection on two-dimensional divergence free vector fields.  The Rossby propagation operator, denoted by $L_R$, is defined by
$$
L_R V =\Ph( b V^\bot).
$$

\end{Def}

\begin{Lem}
Let $\bar v^0_{h}\in L^2(\omega_h)$ be a two-dimensional divergence free vector field, and let $u\in \mathcal C(\bR_+, L^2(\om))$ be the solution of equation \eqref{rescaled} with initial data 
$$
u_{|t=0}=\begin{pmatrix}
 \bar v^0_{h}\\0
         \end{pmatrix},
$$ 
supplemented with the boundary conditions
$$\ba
\d_z u_{h|z=1}=0,\quad u_{3|z=1}=0,\\
\d_z u_{h|z=0}=0,\quad u_{3|z=0}=0.
\ea$$
Then $v=(v_h,0)$, where $u_h$ is a two-dimensional divergence free vector field given by
$$v_h(t)=\frac{1}{{2\pi}} \sum_{k\in\bZ}\int_{\bR}\exp\left(\frac{i\beta t}{\eps} \frac{ k}{|k_h|^2}- \nu_h |k_h|^2 t + ix_h\cdot k_h\right)\hat v^0_{h}(k,\xi_y) \:d\xi_y$$
where $k_h=(k,\xi_y)$ and 
$$
\hat v^0_{h}(k,\xi_y) =\frac{1}{2\pi}\int_{\omega_h} \exp(-i(xk+ y\xi_y))\bar v^0_{h}(x,y)\:dx\:dy,\quad \forall (k,\xi_y)\in\bZ\times \bR.
$$
\label{lem:prop2D}

\end{Lem}

\begin{proof}Let us first prove that the property $\d_z u=0$ is propagated by equation \eqref{NS}.
Using the same arguments as Chemin, Desjardins, Gallagher and Grenier in \cite{CDGG} for classical rotating fluids, one can introduce some kind of Fourier variable with respect to $z$, denoted by $k_3$. Since equation \eqref{NS} is linear, it can be easily checked that there is no resonance between Fourier modes in $k_3$; in other words, since the only Fourier mode at time $t=0$ is $k_3=0$, there is no Fourier mode corresponding to $k_3\neq 0$ for $t>0$, which means exactly that $\d_z v=0$.

We infer that for all $t\geq 0$, $v(t)$ is a two-dimensional vector field which satisfies
\be\begin{aligned}
   \d_t v_h  + \frac{1}{\eps } L_R v_h -\nu_h \Delta_h v_h=0,\quad
\Div_h v_h=0,\\
\d_z P=0,\ v_3=0.
  \end{aligned}
\label{eq:prop2D}\ee
This leads to
$$
v_h(t)=\exp\left(t\left(-\frac{L_R}{\eps} + \nu_h \Delta_h\right)\right)v_{h|t=0}.
$$
Let us now investigate the precise expression of the operator $L_R$. First, since $v_h$ is divergence free, we have, for all $y\in\bR$,
$$
\d_y \int_{\bT}v_2(\cdot, y)=-\int_{\bT}\d_x v_1(\cdot, y)=0.
$$
Consequently, since $v_h\in L^2(\bT\times \bR),$
$$
\int_{\bT} v_2(t,\cdot, y)=0\quad \forall t\geq 0,\ y\in\bR.
$$
Taking the $x$-average of the first component of \eqref{eq:prop2D}, we obtain
$$
\d_t \int_{\bT} v_1 - \nu_h \d_y^2 \int_\bT v_1=0.
$$
This corresponds to the ``stationary part'' of $v^\eps$ in Theorem \ref{thm:propagation}.

Hence Lemma \ref{lem:prop2D} is proved for the Fourier modes such that $k=0$, where $k$ is the Fourier variable associated with $x$. Thus we now focus on the modes such that $k\neq 0$, or, in other words, on initial data such that $\int_\bT \bar v^0_h=0$. For such vector fields, we have, since $v_h\in L^2(\bT\times \bR)$ is divergence free,
$$v_h= \nabla_h^\bot \Delta_h^{-1} \zeta,$$
where 
$$
\zeta(t):=\rot_h v_h =\d_x v_2-\d_y v_1.
$$
On the other hand,
$$
\rot_h (bv_h^\bot))= \Div_h(bv_h)= v_h \cdot \nabla b = \beta v_2.
$$
Gathering the last two inequalities, we infer that
$$
\d_t \zeta+ \frac{\beta}{\eps}\d_x \Delta_h^{-1} \zeta-\nu_h \Delta_h \zeta =0.
$$
In Fourier space, this leads to
$$
\d_t \hat \zeta(k,\xi_y)- i\frac{\beta k}{\eps(|k|^2+ |\xi_y|^2)}\hat \zeta(k,\xi_y)+ \nu_h(|k|^2+ |\xi_y|^2)\zeta =0,
$$
and thus, setting $k_h=(k,\xi_y)$,
\begin{eqnarray*}
 \hat v_h(t,k,\xi_y)&=&- \frac{ik_h^\bot}{|k_h|^2}\exp\left( i\frac{\beta k}{\eps|k_h|^2}t-\nu_h |k_h|^2t \right)\hat \zeta_{|t=0}(k,\xi_y) \\
&=&\exp\left( i\frac{\beta k}{\eps|k_h|^2}t -\nu_h |k_h|^2t\right)\frac{k_h^\bot \cdot \hat v^0_{h}(k,\xi_y) }{|k_h|^2}k_h^\bot.
\end{eqnarray*}
Since $v$ is a two-dimensional divergence free vector field, for all $k_h\in\bZ\times \bR$, we have $k_h\cdot \hat v^0_{h}(k_h)=0$, and thus
$$
\hat v^0_{h}(k_h)= \frac{k_h^\bot \cdot \hat v_h(k_h)}{|k_h|^2}k_h^\bot.
$$
Eventually, we retrieve
$$
 \hat v_h(t,k,\xi_y)=\exp\left( i\frac{\beta k}{|k|^2+ |\xi_y|^2}t -\nu_h |k_h|^2 t\right)v^0_{h}(k,\xi_y)\quad \forall t,k,\xi_y.
$$
Using the Fourier inversion formula, the proof of the Lemma is complete.
\end{proof}

\input 3D-propagation

\section{Derivation of the thermocline}
\label{sec:thermocline}

This section is devoted to the proof of 
Proposition \ref{prop:thermocline}, which relies on classical elliptic arguments. The main difficulty lies in the fact that the equation on $\theta$ is degenerate in the horizontal variables. We first prove the existence of $\bar \theta$, along with some $H^1$ estimates, and then we prove the convergence. 

Throughout the proof, we assume that the wind stress $\sigma$ vanishes at sufficiently high order near $y=0$, so that there is no need for a truncation (see section \ref{sec:BL}) and the function $\ustat$ does not have any singularity.

\vskip1mm

$\bullet$ \textbf{\textit{A priori} estimates on the function $\bt$:}

\vskip1mm

Let $\bt\in L^2(\omega_h, H^1([0,1]))$ be any solution of \eqref{eq:btheta}. Multiplying \eqref{eq:btheta} by $\bt$ and integrating on $\omega$, we obtain
\begin{eqnarray}
\lambda\int |\d_z \bt|^2&=&-\frac{1}{2}\int_{\d \omega}\uint\cdot n_{ \omega} \bt^2	+\lambda \int_{\omega_h}\d_z \bt_{|z=1} \bt_{|z=1}- \lambda \int_{\omega_h}\d_z \bt_{|z=0} \bt_{|z=0}\nonumber\\
&=&-\frac{1}{2}\int_{\omega_h}\uint_{3|z=1}\theta_1^2+ \lambda \int_{\omega_h} \theta_1 \d_z\bt_{|z=1}.\label{eq:dz-bt}
\end{eqnarray}
According to section \ref{sec:int}, we have 
$$
\uint_{3|z=1}=\frac{\d_x\sigma_2}{b} - \d_y \frac{\sigma_1}{b}.
$$
We assume that $\sigma $ is such that the right-hand side belongs to $L^\infty(\omega_h).$ 
We now evaluate $\d_z\bt_{|z=1}:$ we have
\begin{eqnarray}
\lambda \d_z\bt_{|z=1}&=&\lambda \int_0^1\d_{zz}\bt= \int_0^1 \uint\cdot \nabla \bt\nonumber\\
&=&\Div_h(\uint_h\int_0^1\bt) +  \uint_{3|z=1}\theta_1.\label{eq:dz-bt_1}
\end{eqnarray}
Recall that $\uint_h$, defined in section \ref{sec:int}, is independent of $z$, while $\uint_3$ is linear with respect to $z$.
Consequently, the function $\bt$ depends on $x_h$ and $z$, and 
\begin{eqnarray}
\lambda \int_{\omega_h} \theta_1 \d_z\bt_{|z=1}&=&\int_{\omega_h} \theta_1\left(\Div_h(\uint_h\int_0^1\bt) +  \uint_{3|z=1}\theta_1 \right)\nonumber\\
&=&-\int_\omega\bt \uint_h\cdot \nabla_h \theta_1  + \int_{\omega_h}\uint_{3|z=1}\theta_1^2.\label{rhs}
\end{eqnarray}
Using the identity
$$
\bt(\cdot, z)=\theta_1 - \int_z^1\d_z\bt(\cdot, z')\:dz',
$$
we deduce that
\be\label{Poincare_bt}
\|\bt\|_{L^2(\omega)}\leq\|\theta_1\|_{L^2(\omega_h)} +\|\d_z\bt\|_{L^2(\omega)}.
\ee
Gathering \eqref{eq:dz-bt}, \eqref{rhs} and \eqref{Poincare_bt}, we infer that 
\begin{eqnarray*}
\lambda\int_\omega |\d_z \bt|^2	&=& \frac{1}{2}\int_{\omega_h}\uint_{3|z=1}\theta_1^2-\int_\omega\bt \uint_h\cdot \nabla_h \theta_1 \\
&\leq & \frac{1}{2}\| \uint_{3|z=1}\|_{L^\infty(\omega_h)}\|\theta_1\|_{L^2}^2 \\&&+  \|\uint_h\|_{L^\infty}\|\nabla_h\theta_1\|_{L^2}(\|\theta_1\|_{L^2}+\|\d_z \bt\|_{L^2}).
\end{eqnarray*}
Using the Cauchy-Schwarz inequality, we obtain eventually
\begin{eqnarray}\label{est:bt}
\lambda\int_\omega |\d_z \bt|^2	 &\leq &\| \uint_{3|z=1}\|_{L^\infty}\|\theta_1\|_{L^2}^2 +  \|\uint_h\|_{L^\infty}\|\nabla_h\theta_1\|_{L^2}\|\theta_1\|_{L^2} \\&&\nonumber+ \frac{1}{\lambda }\|\uint_h\|_{L^\infty}^2\|\nabla_h\theta_1\|_{L^2}^2.
\end{eqnarray}
Inequalities \eqref{est:bt} and \eqref {Poincare_bt} entail that any solution $\bt$ of \eqref{eq:btheta} is bounded in $L^2(\omega_h, H^1([0,1]))$ by a constant depending only on $\lambda$, $\theta_1$ and $\uint$.

We now derive estimates on the horizontal derivatives in a similar fashion: we have
\be\label{eq:nabla_bt}
-\lambda\d_{zz}\nabla_h\bt + (\uint\cdot\nabla)\nabla_h\bt = - (\nabla_h \uint_h)\cdot \nabla_h \bt - \nabla_h\uint_3\d_z\bt.
\ee
Multiplying the above equation by $\nabla_h\bt$ and integrating by parts, we have, using the boundary conditions,
\begin{eqnarray*}
-\int_\omega \d_{zz}\nabla_h\bt\cdot \nabla_h \bt &=&\int_\omega |\d_z \nabla_h \bt|^2 - \int_{\omega_h} \d_z \nabla_h \bt_{|z=1} \cdot \nabla_h \theta_1\\
&=&\int_\omega |\d_z \nabla_h \bt|^2  + \int_{\omega_h} \d_z  \bt_{|z=1}\Delta_h \theta_1.
\end{eqnarray*}
Using equation \eqref{eq:dz-bt_1}, we express $\d_z\bt_{|z=1}$ in terms of $\bt$ and $\theta_1$. Integrating by parts once again leads to
$$
\left|  \int_{\omega_h} \d_z  \bt_{|z=1}\Delta_h \theta_1 \right|\leq \frac{1}{\lambda}\left(\|\uint_h\|_{L^\infty} \|\bt\|_{L^2} \|\theta_1\|_{H^3} +  \|\uint_3\|_{L^\infty} \|\theta_1\|_{H^2}\|\theta_1\|_{L^2}  \right).
$$
On the other hand, since $\uint_{3|z=0}=0$, we have
\begin{eqnarray*}
2\int_\omega 	\left[ (\uint\cdot\nabla)\nabla_h\bt\right]\cdot \nabla_h\bt&=&\int_{\omega_h}\uint_{3|z=1}|\nabla_h\bt_{|z=1}|^2 -\int_{\omega_h}\uint_{3|z=0}|\nabla_h \bt_{|z=0}|^2\\
&=&\int_{\omega_h}\uint_{3|z=1}|\nabla_h\theta_1|^2.
\end{eqnarray*}
The two terms in the right-hand side of \eqref{eq:nabla_bt} can easily be evaluated in $L^2$ using the estimate on $\d_z \bt$; there remains
$$\lambda\int_\omega |\d_z \nabla_h \bt|^2\leq C + \|\nabla_h \uint_h\|_{L^\infty }\|\nabla_h\bt\|_{L^2}^2,$$
where the constant $C$ depends on $\lambda $, $\|\uint\|_{L^\infty}$ and $\|\theta_1\|_{H^3}$.

Assume that
$$
 \| \nabla_h\uint_h\|_{L^\infty(\omega)} \leq \frac{\lambda}{2};
$$
this assumption is discussed in Remark \ref{rem:thermocline} following Proposition \ref{prop:thermocline}.
Then
$$
\int_\omega |\d_z \nabla_h \bt|^2+ \int_\omega | \nabla_h \bt|^2\leq C,
$$
where the constant  $C$ depends  on $\lambda$, $\theta_1$ and $\uint$.
These estimates easily lead to the existence of a solution $\bt$ of equation \eqref{eq:btheta}; the uniqueness of $\bt$ follows from the estimates above with $\theta_1=0$.
The same method also shows that under  condition \eqref{hyp:u_h_lambda} on $\nabla_h\uint_h$, $D_h^2\bt$ is bounded in $ L^2(\omega_h, H^1([0,1])$. Plugging this estimate back into \eqref{eq:nabla_bt}, we deduce that $\nabla_h \bt\in L^2(\omega_h, H^2[0,1])$, and thus that $\nabla_h \bt$ is bounded in $L^2(\omega_h, W^{1,\infty}([0,1])).$

Concerning the function $\tbl$, the existence and uniqueness are obvious; we have merely
$$
\tbl(x_h,\zeta)=\frac{1}{2\lambda}\nabla_h\theta_1\cdot \sum_\pm (\sigma\pm i\sigma^\bot)\frac{\exp(-\lambda^\pm(x_h)\zeta)}{(\lambda^\pm(x_h))^3}.
$$

\vskip1mm

$\bullet$ \textbf{Proof of convergence:}

\vskip1mm

We construct an approximate solution of \eqref{eq:theta} as follows: we set
$$
\tapp(x_h,z)=\bt(x_h,z) + \eps \tbl\left( x_h,\frac{1-z}{\eps} \right) + \eps  \tilde \theta(x_h,z),
$$
where the function $ \tilde \theta$ is defined by
$$
 \tilde \theta(x_h,z)=(z-1)\frac{1}{\eps}\d_\zeta\tbl_{|\zeta=\frac{1}{\eps}}(x_h) - \tbl_{|\zeta=0}(x_h).
$$
Notice that by construction, 
$$
\d_z \tapp_{|z=0}=0,\quad \tapp_{|z=1}=\theta_1.
$$
Moreover, using the definition on $\tbl$, it is easily proved that $\tilde \theta=O(1)$ in $W^{2,\infty}(\omega)$ (provided the stress $\sigma$ is smooth and vanishes at a sufficiently high order near $y=0$).

Consequently, 
\begin{eqnarray}
\label{eq:tapp}&&-\lambda \d_{zz} \tapp - \lambda \eps^2\Delta_h \tapp + \ustat \cdot \nabla \tapp\\
\nonumber&=&\ubl_h\cdot \nabla_h (\bt - \theta_1) - \lambda \eps^2 \Delta_h \bt + \ubl_3\d_z\bt + \vint\cdot \nabla \bt\\
\nonumber&&-\lambda \eps^3 \Delta_h \tbl\left( x_h,\frac{1-z}{\eps} \right)+ \eps \ustat\cdot \nabla  \tbl\left( x_h,\frac{1-z}{\eps} \right)\\
\nonumber&&-\lambda \eps^3 \Delta_h \tilde \theta+ \eps \ustat\cdot \nabla  \tilde \theta.
\end{eqnarray}
According to the results of sections \ref{sec:BL} and \ref{sec:int}, we have
$$\ba
\|\ubl_3\|_{L^2}=O(\sqrt{\eps}),\quad \|\vint\|_{L^2}=o(\eps),\\
\eps \|\ustat_h\|_{L^\infty}, \|\ustat_3\|_{L^\infty}=O(1),\quad \eps \|\ustat\|_{L^2}=O(\sqrt{\eps}).
\ea
$$
These estimates, together with the ones derived above on $\bt$, enable us to bound all the terms in the right-hand side of \eqref{eq:tapp}, except for the first one. Using Hardy's inequality, we have
\begin{eqnarray*}
&&\left\| \ubl_h\cdot \nabla_h (\bt - \theta_1)\right\|_{L^2(\omega)}\\
&\leq & \left\| (1-z)\ubl_h  \right\|_{L^\infty(\omega_h, L^2([0,1]))} \left\| (1-z)^{-1} \nabla_h (\bt - \theta_1) \right\|_{L^2(\omega_h, L^\infty([0,1]))}\\
&\leq & C\sqrt{\eps} \left\| \d_z \nabla_h (\bt - \theta_1) \right\|_{L^2(\omega_h, L^\infty([0,1]))}.
\end{eqnarray*}
Thus $\tapp$ is an approximate solution of \eqref{eq:theta}, with an error term $o(1)$ in $L^2(\omega)$. As a consequence, $\theta-\tapp$ satisfies
$$
\ba
-\lambda \d_{zz} (\theta-\tapp) - \lambda \eps^2\Delta_h (\theta-\tapp) + \ustat \cdot \nabla (\theta-\tapp)=o(1),\\
(\theta-\tapp)_{|z=1}=0, \quad\d_z(\theta-\tapp)_{|z=0}=0.
\ea
$$
Multiplying the above equation by $\theta-\tapp$ and using the Poincar\'e inequality, we prove that 
$$
\left\|\d_z (\theta-\tapp)  \right\|_{L^2(\omega)}=o(1),
$$
and thus the Proposition is proved. 

\section*{Acknowledgements}

This work received the support of  the Agence Nationale de la Recherche (project ANR-08-BLAN-0301-01).

\end{document}

%% file: 3D-propagation.tex
\def\eqdefa{\buildrel\hbox{\footnotesize def}\over =}

\section{Three-dimensional propagation}
\label{sec:3D}

\subsection{Remarks about the qualitative behaviour of three-dimensional waves}$ $

We are now interested in waves having vertical oscillations, that is in the solutions to
\begin{equation}
\label{3D}
\begin{aligned}
\d_t u +\frac1\eps \beta y u^\perp +\left( \begin{matrix} \nabla_h p\\ \frac1{\eps^2} \d_z p \end{matrix}\right) -\nu_h \Delta_h u -\eps \d_{zz }u =0,\\
\nabla \cdot  u=0,\\
\d_z u_{h|z=0}=\d_z u_{h|z=1}=0,\quad u_{3|z=0} =u_{3|z=1} =0,
\end{aligned}
\end{equation}
having zero average with respect to $z$.

Once again, we introduce a kind of Fourier variable with respect to $z$ (see \cite{CDGG}), denoted by $k_3$, which here is different from zero. The Fourier variable associated with the first coordinate $x$ is still denoted by $k$.

If $\nu_h$ is sufficiently small, we then expect the main dynamics to be given by the Poincar\'e propagation operator
\begin{equation}
\label{LP-def}
L_{{P}}u = \beta y u^\perp +\left( \begin{matrix} \eps \nabla_h p\\ \frac1{\eps} \d_z p \end{matrix}\right)
\end{equation}
where $p$ is such that both the incompressibility constraint and the boundary condition are satisfied.

$\bullet$ A very rough analysis shows that {\bf fast oscillations with respect to $y$ should appear} for times greater than $\eps$. Indeed, as long as the solution $(u,p)$ to
$$\eps \d_t u + L_{{P}}u = 0$$
depends slowly on $y$, the pressure which satisfies
$$-(\d_{xx} +\d_{yy} p+\frac{1}{\eps^2}\d_{zz}) p = -\frac1\eps \beta y \d_x u_2+\frac1\eps \d_y(\beta y u_1)$$
can be approximated in the following way
$$\hat p=\frac \eps {k_3^2}\left(-ik \beta y \hat u_2+\d_y(\beta y \hat u_1)\right)=O(\eps) .$$
In particular, at leading order, the singular penalization behaves as in the compressible case
$$
(L_{ P}u )_h\sim \beta y u^\perp_h
$$Plugging this Ansatz in the evolution equation leads to
$$u_h \sim \sum_\pm u_h^{0,\pm} \exp \left(\pm i{\frac{\beta y t }{\eps} }\right),$$
which is relevant only for very small times, but indicates that a fast dependence with respect to $y$ can be expected.

\medskip
$\bullet$ On the other hand, we do not expect $(u,p)$ to behave as a function of $y/\eps$ only. Such a property, together with usual integrability conditions, would indeed imply that the solution $(u,p)$ concentrates on small times in the vicinity of $y=0$. As previously, a rough analysis based on the change of variable $Y=y/\eps$ and on some asymptotic expansion of $L_{{P}}u$
$$
\widehat{ L_{{P}} u}\sim (0, ik_3(k_3^2-\d_{YY})^{-1} (\beta Y \d_Y \hat u_1))$$
shows that ``concentrated functions" are not stable under the penalization $L_{{P}}$.

The mechanism we want to study \textbf{ involves therefore both scales $y$ and $y/\eps$}, and results from a balance between rotation and vertical oscillations, which is the main novelty here. Note indeed that previous works on rotating fluids consider either the case when the effect of rotation is dominating (macroscopic layer of fluid) \cite{CDGG} or the case when vertical oscillations hold on very small scales and can be averaged (shallow water approximation) \cite{GSR}.

Semiclassical analysis seems therefore to be the relevant tool to study this problem, insofar as it allows to separate both scales in a systematic way.

\medskip
$\bullet$ Note finally that, if the horizontal viscosity is such that $\nu_h \gg\eps^2$, then because of the small scale in $y$, we expect {\bf all the energy to be dissipated on a small time interval}, leading to some boundary layer effect (see the discussion in paragraph \ref{ssec:viscosity}).

In order to exhibit a non trivial propagation, we will assume in all the sequel that $$\nu_h =o(\eps^2).$$
We therefore start with the study of the 3D propagation without dissipation. We will then check \textit{a posteriori} that the viscous dissipation introduces only small error terms for any finite time.

\subsection{Semiclassical analysis of the three-dimensional propagation}$ $

In order to study the propagation of energy by 3D waves, a natural idea is then to get a polarization of Poincar\'e waves, i.e. to obtain a diagonalization of the system
$$\eps \d_t u + L_{{P}}u = 0$$
in the limit $\eps \to 0$.
We first use the incompressibility constraint to rewrite the propagator in the form of a $2\times 2$ matrix of pseudo-differential operators. We indeed have
$$-\Delta_\eps p := -\eps^2 (\Delta_h+\frac{1}{\eps^2} \d_{zz}) p=-\eps \beta y \d_x u_2 +\eps \d_y (\beta y u_1)$$
from which we deduce that
$$\eps \d_t u_h + \left( \begin{matrix} -\eps^2 \d_x \d_y \Delta_\eps^{-1} (\beta y \cdot) &-\beta y \cdot  -\eps^2\d^2_{xx} \Delta_\eps^{-1} (\beta y\cdot) \\
\beta y \cdot  -\eps^2\d^2_{yy} \Delta_\eps^{-1} (\beta y \cdot) & \eps^2 \d_x \d_y \Delta_\eps^{-1} (\beta y \cdot)
\end{matrix} \right)u_h =0.$$
Our first goal is then to perform a suitable change of variables leading to 
 $$\eps \d_t v + \left( \begin{matrix} H_\eps^+(\d_x,\d_z, y,\eps \d_y) & 0\\0&H_\eps^-(\d_x,\d_z, y,\eps \d_y)\end{matrix} \right) v =O(\eps^\infty).$$

In all the sequel, for the sake of simplicity, we will consider a single Fourier mode in $(x,z)$, and denote by $(k,k_3) \in \Z \times \Z^*$ the associated wavenumber. Any solution is  indeed a superposition of such waves. We will denote abusively $H_\eps^\pm(k,k_3,y,\eps \d_y)$ the  Fourier transform of $H_\eps^\pm (\d_x,\d_z, y,\eps \d_y)$.

\medskip
We are then brought back to study the propagation of waves by the  scalar pseudo-differential operator $H_\eps^\pm(k,k_3,y,\eps \d_y)$, which can be done for instance using classical results on the Wigner transform. For such scalar skew-symmetric pseudo-differential operators, we indeed know \cite{GMMP} that energy is propagated according to the hamiltonian transport equations
$$\d_t f +\{ h^\pm,f\} =0,$$
where $h^\pm(k,k_3,y,\xi) $ is the semiclassical principal symbol of $H_\eps^\pm(k,k_3,y,\eps \d_y)$.

Note that the time scale over which one has a macroscopic propagation of the energy is inversely proportional to the size of the oscillations. Such a property can be seen very simply on equations with constant coefficients 
$$\eps \d_t v+ h(\eps \d_y) v =0$$
remarking that the group velocity 
$$ \frac{d}{dk_2} \frac1\eps h(i\eps k_2)$$
has a finite limit as $\eps k_2 \to \xi$.

\medskip
What we are finally able to establish is the following Proposition
\begin{Prop}\label{3D-prop}
Let $u^0\in L^2(\omega)$ be a  compactly supported divergence free vector field such that $\int u^0 dz =0$ , and let $u\in \mathcal C(\bR_+, L^2(\om))$ be the solution of equation \eqref{3D} with initial data 
$u^0$.\\
Then the  $L^2$ norm of $u_h(t)$ on any fixed compact  converges to 0 as $t\to \infty$. \end{Prop}

In other words, 3D waves are dispersive, but only on times of order 1. 
Note that, in the case of a macroscopic layer of fluid, the velocity group of Poincar\'e waves is much larger (typically of order $1/\eps$); see for instance \cite{CDGG,GSR-handbook}.

Furthermore the vertical component $u_3$ of the velocity will not remain bounded, as is usually claimed in formal derivations leading to shallow water models.

\subsection{Reduction to a scalar situation}$ $

The first step of the proof follows a method initiated in \cite{CGPSR}.

$\bullet$ We first compute {\bf a kind of characteristic polynomial} for the matrix of pseudo-differential operators
$$\left(  \begin{matrix}
\eps (-\Delta_\eps)^2)^{-1} ik \eps \d_y (\beta y \cdot) & -\beta y + \eps (-\Delta_\eps)^{-1} k ^2\beta y\\
\beta y +\eps \d_y(-\Delta_\eps)^{-1} \eps \d_y (\beta y \cdot)  & \eps \d_y (-\Delta_\eps)^{-1} (i\eps k \beta y\cdot )\end{matrix}\right)$$
A simple way to obtain a scalar equation is to proceed by linear combination and substitution.

Because the solution is expected to depend both on $y$ and $y/\eps$ (whatever the initial data), $\eps \d_y$ is a $O(1)$ operator like multiplication by any function of $y$. We then apply usual rules of semiclassical analysis~:
$$\eps \d_y =O(1), \quad y=O(1),$$
and any commutator has smaller order
$$[\eps \d_y, y] = O(\eps).$$
Keeping only leading order terms, we get
$$\begin{aligned}
i\tau \hat u_1-\beta y \hat u_2 = O(\eps),\\
\beta y\hat u_1  +\eps \d_y(k_3^2-(\eps \d_y)^2)^{-1} \eps \d_y (\beta y \hat u_1)+i\tau \hat u_2=O(\eps)
\end{aligned}
$$
so that 
$$ \beta^2 y^2 \hat u_2 +\eps \d_y(k_3^2-(\eps \d_y)^2)^{-1} \eps \d_y (\beta^2 y^2 \hat u_2)-\tau ^2\hat u_2=O(\eps)$$
or equivalently
\begin{equation}
\label{dispersion-relation}
k_3^2(\beta y)^2 \hat u_2  -\tau^2(k_3^2-(\eps \d_y)^2)  \hat u_2=O(\eps)
\end{equation}
since commutators provide higher order terms with respect to $\eps$. Note  that one can also compute an exact pseudodifferential relation (which is actually a polynomial of degree 6 with respect to $\tau$)  by keeping all the terms
\begin{equation}
\label{dispersion-relation-ex}
P(\eps,y,\eps \d_y,\tau)\hat u_2 =0.
\end{equation}

Note that, contrarily to \cite{CGPSR}, as we will only consider times of order $1$, we do not need to compute subsymbols, so that we could also proceed directly using symbolic calculation and diagonalize the matrix
$$\left(  \begin{matrix}
0 & -\beta y \\
\beta y -{\xi ^2 \beta y \over k_3^2+\xi^2}   &0\end{matrix}\right).$$

Anyway, we expect the roots to the following  polynomial to play a special role in the propagation~:
\begin{equation}
\label{dispersion-relation-symbol}
P(0, y,\xi,\tau) =k_3^2 (\beta y)^2 -(k_3^2+\xi^2)\tau^2
\end{equation}

\medskip
$\bullet$ We can actually prove that there exist pseudo-differential operators $H_\eps^\pm$ with principal symbols
$$h^\pm = \pm  \sqrt{ (k_3\beta y)^2\over k_3^2+\xi^2}$$
such that 
$\epsilon\partial_t \mu^\pm=iH_\eps^\pm \mu^\pm $ implies that
$$
v^\pm :=\left(\begin{array}{c}(iH_\eps^\pm)^{-1} (\beta y \cdot) \\
\mathbb Id\\ -{k\over k_3} (iH_\eps^\pm)^{-1} (\beta y \cdot)+{i\over \eps k_3} (\eps \d_y\cdot ) \end{array}\right)\mu^\pm \mbox{ satisfies (\ref{3D}) up to } O(\epsilon^\infty),
$$
where $\mu^+, \mu^-$ are scalar functions.

This result is actually a variant of the main Lemma in \cite{CGPSR}. (Indeed the exact dispersion relation depends here explicitly on $\eps$.)

\begin{Lem}\cite{CGPSR}
Let $P_\eps=P(\eps, y,\xi,\tau)$ be a smooth function such that $\partial_\tau P_{0| P=0}\neq 0$, and let $h=h(y,\xi)$ be any continuous root of $$P(0,y, \xi,h(y,\xi))=0.$$

Then there exists a pseudo-differential operator $H_\eps=H_\eps(y,-i\epsilon\partial_y)$ with principal symbol $h(y,\xi)$ such that:
\begin{equation}
\label{theother}
H_\eps\psi=\tau\psi\ \Longrightarrow\ {\bf P}_{\eps,\tau}\psi=O(\epsilon^\infty)
\end{equation}
where ${\bf P}_{\eps,\tau}$ is a pseudo-differential operator of full symbol $P(\eps,y,\xi,\tau)$. 

\end{Lem}

The proof of this lemma relies on pseudo-differential functional calculus, and uses various quantifications to make the computations as simple as possible. For the sake of completeness, we recall here the main arguments,  but refer to \cite{CGPSR} for details.

 At first order, we  have
$$\begin{aligned}
{\bf P}_{\eps,\tau} \psi\equiv &\int e^{i\frac{\xi(y-y')}\epsilon}P(\eps,y,\xi,H_\eps(y',-i\epsilon\partial_y))\psi(y')\frac{d\xi dy'}\epsilon\\=&
\int e^{i\frac{\xi(y-y')}\epsilon}e^{i\frac{\xi'(y'-y'')}\epsilon}P(\eps, y,\xi,h(y'',\xi'))\psi(y'')\frac{d\xi d\xi'dy'dy''}{\epsilon^2}\\=&\int e^{i\frac{\xi(y-y')}\epsilon}P(\eps, y,\xi,h(y',\xi))\psi(y')\frac{d\xi dy'}\epsilon\end{aligned}
$$
So the principal symbol of ${\bf P}_{\eps,\tau} $ is $P(0, y,\xi,h(y,\xi))$ which, by assumption, is $0$.

 For the $\epsilon^\infty$ result, it is enough to repeat the same argument  with  $h_\epsilon\sim h+\sum \epsilon^k h_k$. We obtain
$$
P(\eps, y,\xi, h_\epsilon)+\sum_{k\geq 1}\epsilon^kQ_k(h,\dots, \partial_y^l\partial_\xi^m\ h_\epsilon)=0,
$$ 
that can be solved recursively under the condition $\partial_\tau P_{0| P=0}\neq 0$.

\medskip
$\bullet$ We further obtain a decomposition of any initial data on the eigenstates of the scalar propagators $H_\eps^\pm$.

For all $ u_{h}^0$, there exist $  \mu_\eps^{0,\pm}$ such that:
\begin{eqnarray*}
u_{0,h} &=&\sum_{j} \left(\begin{matrix} -\left( \beta y -\eps \d_y\hat \Delta_\eps^{-1} \eps \d_y (\beta y \cdot)\right)^{-1} (iH_\eps^j - \eps \d_y \hat \Delta_\eps^{-1} (i\eps k \beta y\cdot ) \\\mathbb Id\end{matrix} \right)\mu_\eps^{0,j }\\&&+O(\epsilon^\infty)\\
& =:&\sum_{\pm} {\mathbb Q}_\eps^j \mu_\eps^{0,j }+O(\epsilon^\infty) .	
\end{eqnarray*}

where $\hat \Delta_\eps := \eps^2 \d_{yy}^2-\eps^2 k^2-k_3^2$. 
The vertical component is then entirely determined by the divergence-free condition.

To prove this result,  one first remarks that the  leading order symbol of the matrix $(
{\mathbb Q}_\eps^+  \, {\mathbb Q}_\eps^- )$, namely
$$\left( \begin{matrix} -{i\sqrt{k_3^2+\xi^2}\over |k_3| \hbox{sgn}(y)} &  {i\sqrt{k_3^2+\xi^2}\over |k_3| \hbox{sgn}(y)}\\
1 & 1\end{matrix}\right)$$ 
 is invertible.
 
The inversion of the matrix $(
{\mathbb Q}_\eps^+  \, {\mathbb Q}_\eps^- )$ can then be done symbolically at any order.

\subsection{Dispersion of energy}$ $

Standard arguments of semiclassical analysis allow then to control the propagation of energy for the scalar equations
$$ \eps \d_t \mu_\eps^\pm + i H_\eps^\pm \mu_\eps^\pm=0$$

$\bullet$ Because $L_{{P}}$ is skew-symmetric (in some weighted $L^2$-space), we have a uniform control on the $L^2$ norm of $u_h$ 
$$ \| u_h\|_{L^2(\omega)}^2+ \eps^2 \| u_3\|_{L^2(\omega)}^2 = \| u_h^0\|_{L^2(\omega)}^2+ \eps^2 \| u_3^0\|_{L^2(\omega)}^2 $$
These uniform \textit{a priori} estimates allow to establish the convergence of the remainders in the equations for the Wigner transforms
$$f^\pm_\eps (t, y,\xi):=\frac1\pi  \int e^{2i \xi y'} \mu_\eps^\pm (y-\eps y') \bar \mu_\eps^\pm(y+\eps y') dy'$$

We therefore have
$$\d_t f_\eps^\pm +\{h^\pm, f_\eps^\pm\} =O(\eps).$$

For detailed computations leading to that estimate, we refer for instance to \cite{lions-paul} or \cite{GMMP}~:

\begin{Lem}\cite{GMMP}
Let $\mu^{0,\pm}$ be any fixed function of $L^2$ (non-oscillatory).

Assume that 
\begin{itemize}
\item $iH_\eps^\pm$ is self-adjoint on $L^2$;
\item  there exists $\sigma \in \bR$ such that $ H_\eps^\pm$ is of order $\sigma$ uniformly as $\eps \to 0$;
\item the Weyl symbol of $H_\eps^\pm$ satisfies
$$h^\pm_\eps = h^\pm +\eps h^\pm_1 +o(\eps)\hbox{ uniformly in }C^\infty_{loc}.$$
\end{itemize}

Then the Wigner transform $f_\eps^\pm(t,y,\xi)$ of $\mu_\eps^\pm(t)$ converges locally uniformly in $t$ to the continuously $t$-dependent positive mesure $f^\pm$, solution to
$$\d_t f ^\pm+\{h^\pm, f^\pm\} =0\,.$$
\end{Lem}

In other words, the energy associated to the $\pm$ mode is transported along the characteristics of the hamiltonian $h^\pm$~:
\begin{equation}
\label{car}
\begin{aligned}
{dY^\pm \over dt} = {\d h^\pm\over \d \xi}(Y^\pm,\Xi^\pm),\\
{d\Xi^\pm\over dt}=-{\d h^\pm\over \d y}(Y^\pm,\Xi^\pm)\,.
\end{aligned}
\end{equation}

\medskip
$\bullet$ The previous 1D hamiltonian systems are of course integrable. The bicharacteristics are indeed included in the level lines of $h^\pm$, which are hyperbola as shown in Figure 2.
\begin{figure} [htbp] %  figure placement: here, top, bottom, or page
   \centering
   \includegraphics[width=4.9in]{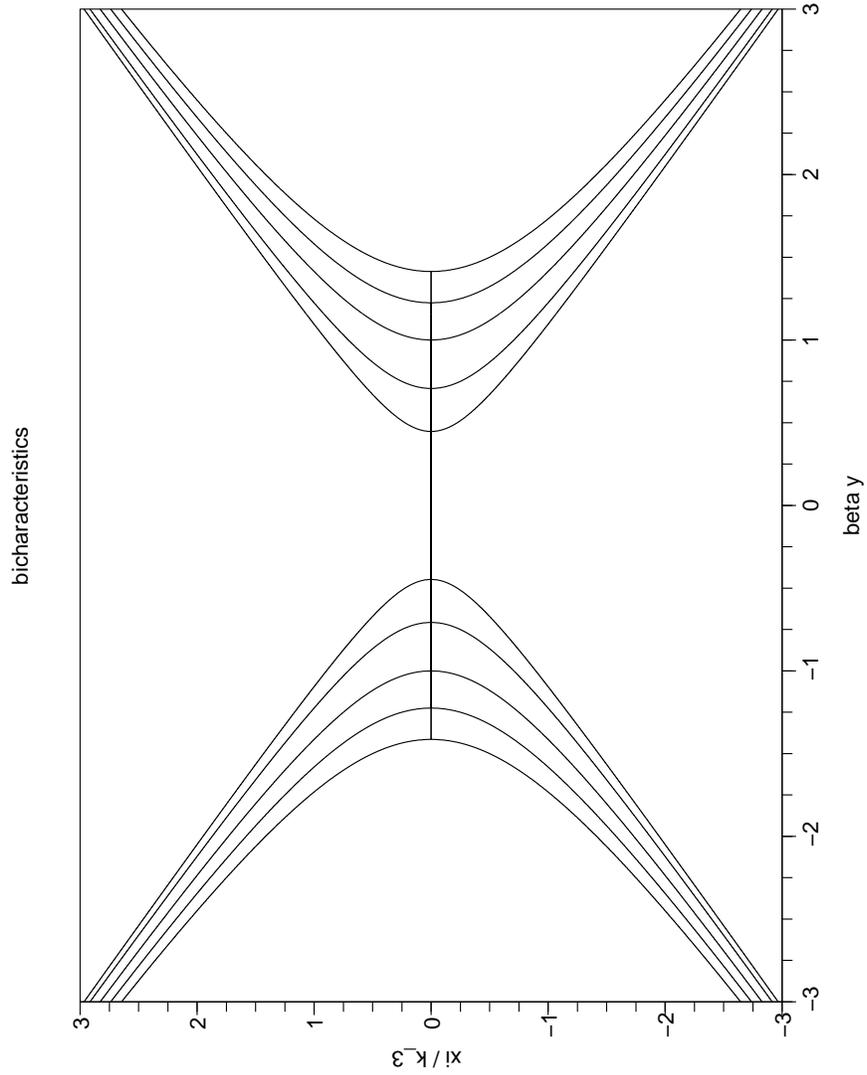} 
   \caption{Bicharacterictics associated to Poincar\'e waves}

\label{fig2}

\end{figure}

A rapid inspection of the large time asymptotics show that trajectories cannot be trapped in some compact. This would indeed imply that there exists either some stationary point or some turning point. But $\Xi(t)$ is a monotonic function
$$
{d\Xi\over dt}=\mp {\beta k_3 \hbox{sgn}(Y(t)) \over \sqrt{k_3^2+\Xi^2(t)}}$$
which converges necessarily to infinity.

For any fixed compact, we can even get an explicit estimate of the exit time since 
$$|\Xi(t)-\Xi_0|\geq  \beta t,$$
from which one deduces a similar estimate for $Y(t)$~:
$$|Y(t)| \geq {h_0\over \beta k_3} \sqrt{k_3^3 + (\beta t -|\Xi_0|)^2}$$
Note that, since the initial data $u_0$ we consider is supposed to depend only on the slow variable $y$, all bicharacteristics we are interested in satisfy $\Xi_0 =0$.

By definition of the wavefront set, we finally obtain Proposition \ref{3D-prop}.

\bigskip
\begin{Rmk}
The qualitative behaviour of Rossby and Poincar\'e waves obtained here, i.e. in the case of a thin layer of fluid with rigid lid, is very different from the one exhibited in shallow water approximations (see \cite{CGPSR}). Note that, in both cases, Rossby waves are easily identified because they are directly linked to the inhomogeneity of the Coriolis force, in particular they always propagate eastwards.

Here the energy associated to Poincar\'e waves propagates much slower than the energy associated to Rossby waves. The point is that fast oscillations with respect to latitude $y$, which are generated spontaneously for vertical modes but not for purely 2D Rossby waves, slow down the propagation. 
Maybe it would be physically relevant to consider initial data that depend already on the fast variable $y/\eps$.

The other  point which should be discussed is the influence of the free-surface. But, at the present time, we have no convenient mathematical tool to study the propagation of waves in such a complex geometry.
\end{Rmk}

\subsection{Influence of the viscosity}$ $
\label{ssec:viscosity}
In the case when $\nu_h = o(\eps^2)$, an easy computation based on the energy estimate shows that the viscous dissipation does not modify the propagation for finite times.

\medskip
More generally, we could extend the previous study considering the whole viscous Poincar\'e propagation operator
\begin{equation}
\label{LPviscous-def}
L_{{P}}u = \beta y u^\perp +\left( \begin{matrix} \eps \nabla_h p\\ \frac1{\eps} \d_z p\end{matrix}\right) -\nu_h \Delta _h u -\eps \d_{zz} u
\end{equation}
where $p$ is such that both the incompressibility constraint and the boundary condition are satisfied.

The diagonalization process is of course unchanged since the dissipation operator is scalar.
The only difference is therefore that one has now to control the propagation of energy for the scalar equations
$$ \eps \d_t \mu_\eps^\pm + i H_\eps^\pm \mu_\eps^\pm-\nu_h \Delta _h \mu_\eps^\pm -\eps \d_{zz} \mu_\eps^\pm=0.$$
A standard computation  (reported for instance in Proposition 1.8 of \cite{GMMP}) shows that   the Wigner transform then satisfies the following damped transport equation 
$$\d_t f _\eps^\pm+4 {\nu_h\over \eps^2} |\xi|^2 f_\eps^\pm+\{h^\pm, f_\eps^\pm\} =o(1)\,.$$
(Note that the
 symmetric part of the propagator occurs at leading order in $\eps$, which can be seen by easy symmetry considerations.)
 
 We then deduce that
 \begin{itemize}
 \item if $\nu_h \ll \eps^2$, the energy is propagated according to the bicharacteristics associated to $h^\pm$, as stated in Proposition \ref{3D-prop} ;
 
 \item if $\nu_h \gg\eps^2$, the energy contained initially in the Poincar\'e modes is dissipated on a very short time, leading to some initial layer phenomenon ;
 
 \item if $\nu_h\sim \eps^2$, the dynamics is a combination of both phenomena, as shown by Duhamel's formula
 $$f^\pm(t,Y^\pm(t,y,\xi),\Xi^\pm((t,y,\xi)) \exp\left( 4{\nu_h \over \eps^2}|\Xi(t,y,\xi)|^2t\right)= f^0(t,y,\xi) \,.$$
Note in particular that the energy associated to Poincar\'e modes has  a super exponential decay, since $|\Xi(t,y,\xi)|\to \infty$ along any trajectory.
 \end{itemize}